\documentclass[10pt,a4paper,fleqn]{article}
\usepackage[latin1]{inputenc}
\usepackage{amsmath}
\usepackage{amsfonts}
\usepackage{amssymb}
\usepackage{amsthm}
\usepackage{makeidx}
\usepackage{color}

\title{Bessel processes and hyperbolic Brownian motions stopped at different random times}

\numberwithin{equation}{section}
\newtheorem{os}{Remark}[section]
\newtheorem{te}{Theorem}[section]
\newtheorem{lem}{Lemma}[section]

\author{D'OVIDIO Mirko\footnote{Dipartimento di Statistica, Probabilit\`a e Statistica Applicata, ''Sapienza'' University of Rome, P.le Aldo Moro n. 5, 00185 Rome (Italy), tel: +390649910499, e-mail:mirko.dovidio@uniroma1.it} \and ORSINGHER Enzo\footnote{Corresponding author: Dipartimento di Statistica, Probabilit\`a e Statistica Applicata, ''Sapienza'' University of Rome, P.le Aldo Moro n. 5, 00185 Rome (Italy), tel: +390649910585, e-mail:enzo.orsingher@uniroma1.it}}

\allowdisplaybreaks

\begin{document}
\maketitle

\begin{abstract}
Iterated Bessel processes $R^\gamma(t)$, $t>0$, $\gamma>0$ and their counterparts on hyperbolic spaces, i.e. hyperbolic Brownian motions $B^{hp}(t)$, $t>0$ are examined and their probability laws derived. The higher-order partial differential equations governing the distributions of $I_R(t)=\, _1R^\gamma( _2R^\gamma(t))$, $t>0$ and $J_R(t) =\, _1R^\gamma(| _2R^\gamma(t)|^2)$, $t>0$ are obtained and discussed. Processes of the form $R^\gamma(T_t)$, $t>0$, $B^{hp}(T_t)$, $t>0$ where $T_t=\inf\{ s: B(s)=t \}$ are examined and numerous probability laws derived, including the Student law, the arcsin laws (also their asymmetric versions), the Lamperti distribution of the ratio of independent positively skewed stable random variables and others. For the process $R^{\gamma}(T^\mu_t)$, $t>0$ ( where $T^\mu_t = \inf \{ s: B^\mu(s)=t\}$ and $B^\mu$ is a Brownian motion with drift $\mu$) the explicit probability law and the governing equation are obtained. For the hyperbolic Brownian motions on the Poincar\'e half-spaces $H^+_2$, $H^+_3$ we study $B^{hp}(T_t)$, $t>0$ and the corresponding governing equation. Iterated processes are useful in modelling motions of particles on fractures idealized as Bessel processes (in Euclidean spaces) or as hyperbolic Brownian motions (in non-Euclidean spaces).
\end{abstract}


\textbf{AMS:} Primary 60J65, 60J60, 26A33.

\section{Introduction}
The analysis of the composition of different types of stochastic processes has recently received a certain attention with the publication of a series of papers (see for example \cite{BMN09}, \cite{BM01}, \cite{AZ01}, \cite{DO}). The prototype of these composed processes is the iterated Brownian motion whose investigation was started in the middle of the '90s. Beside  the distributional properties of the composed processes much work was done in order to derive the equations governing their probability laws. It was found that these processes are related both to fractional equations and to higher-order equations as is the case of iterated Brownian motion. The core of this paper considers the Bessel process $R^\gamma_x(t)$, $t>0$ started at point $x \in [0, \infty)$ and with parameter $\gamma>0$ at different random times. We firstly study the process $I_R(t)=R_1^{\gamma}(R_2^{\gamma}(t))$, $t>0$ where $R_1^\gamma$, $R_2^\gamma$ are independent Bessel processes with the same parameter $\gamma$. This is equivalent to studying the Bessel process $R^\gamma$ at a random time which is represented by an independent Bessel process. Iterated processes have proved to be suitable for describing the motion of gas particles in cracks (or fractures). For the iterated Brownian motion this is considered in DeBlassie's paper \cite{DB04} but a similar interpretation can be given to iterated processes obtained by composing Bessel processes (this is the case here) or fractional Brownian motions (see \cite{DO}). The law of $I_R(t)$, $t>0$ is expressed in terms of Fox functions and possesses a Mellin transform equal to
\begin{equation}
E \left\lbrace I_R(t) \right\rbrace^{\mu -1} = (2^3 t)^\frac{\mu -1}{4} \frac{\Gamma\left( \frac{\gamma +\mu -1}{2} \right) \Gamma\left( \frac{2\gamma + \mu -1}{4} \right)}{\Gamma^2\left( \frac{\gamma}{2} \right)}, \quad \Re\{\mu \}> 1-\gamma , \, t>0 
\end{equation} 
with $\gamma >0$ and for $\mu=m+1$ produces the $m$-th order moments of $I_R(t)$, $t>0$
\begin{equation}
E \left\lbrace I_R(t) \right\rbrace^{m} = (2^3 t)^\frac{n}{4} \frac{\Gamma\left( \frac{\gamma +m}{2} \right) \Gamma\left( \frac{2\gamma + m}{4} \right)}{\Gamma^2\left( \frac{\gamma}{2} \right)}, \quad m=1,2, \ldots .
\end{equation} 
We are able to present the p.d.e. satisfied by the distribution of $I_R(t)$, $t>0$ which differs for $\gamma>1$ and $\gamma \leq 1$ because in the latter case an impulse delta function appears as in the iterated Brownian motion. The equation we obtained reduces to the fourth-order heat-equation
\begin{equation}
\frac{\partial q}{\partial t} = \frac{1}{2^3} \frac{\partial^4 q}{\partial x^4} + \frac{1}{2 \sqrt{2 \pi t}} \frac{d^2 \delta(x)}{d x^2} 
\end{equation}
for $\gamma=1$. A related process considered in Section \ref{seZ} is 
\begin{equation}
J_R(t) = R^\gamma_1(| R^\gamma_2(t) |^2), \quad t>0
\label{WeR}
\end{equation}
where $R_1^\gamma$, $R^\gamma_2$ are independent Bessel processes starting at the origin. The probability density of \eqref{WeR} can be expressed in closed form as
\begin{equation}
q(r,t) = Pr\{ J_R(t) \in dr \} / dr = \frac{2^2 r^{\gamma-1}}{2^\gamma t^{\gamma/2} \Gamma^2\left( \frac{\gamma}{2} \right)} K_0\left( \frac{r}{\sqrt{t}} \right), \quad r,t>0
\label{WeReq}
\end{equation}
where $K_0$ is the modified Bessel function of order zero (see  \cite[formula 3.478]{GR}). The equation corresponding to \eqref{WeReq} has the form 
\begin{equation}
\frac{\partial q}{\partial t} = -\frac{1}{2} \left\lbrace r \frac{\partial^3 q}{\partial r^3} + 2 (2 - \gamma) \frac{\partial^2 q}{\partial r^2} + \frac{(\gamma -1)^2}{r} \frac{\partial q}{\partial r} - \frac{(\gamma -1)^2}{r^2}q \right\rbrace, \quad r,t>0 
\end{equation}
and includes the equations governing the process $B_1(|B_2(t)|^2)$, $t>0$, for $\gamma=1$ (and coincides with 3.16 of \cite{DO} for $H=1/2$). Interesting results can be obtained by considering the Bessel process $R^\gamma(t)$, $t>0$ stopped at the first-passage time $T_t$, $t>0$ of an independent Brownian motion. Processes stopped at different types of random times can be viewed as processes with a new clock which is regulated by an independent Brownian motion $B$. The r.v.'s $T_t=\inf \{ s :\, B(s)=t \}$ tells the time at which the Bessel process must be examined. This means that the clock considered below is timed by an independent Brownian motion. Therefore $R^\gamma(T_t)$, $t>0$ represents a motion where accelerations and decelerations of time occur continuously. We show that the distribution function of $R^\gamma(T_t)$, $t>0$ reads
\begin{equation}
Pr \{ R^\gamma(T_t) \in dr \} /dr = \frac{2t r^{\gamma -1}}{\sqrt{\pi}} \frac{\Gamma\left( \frac{\gamma + 1}{2} \right)}{\Gamma\left( \frac{\gamma}{2} \right) }  \frac{1}{(r^2 + t^2)^\frac{\gamma +1}{2}}, \quad r,t>0 .
\end{equation}
We show also that 
\begin{equation}
Pr\left\lbrace \frac{1}{R^\gamma(T_t)} \in dr \right\rbrace / dr= \frac{2t}{\sqrt{\pi}} \frac{\Gamma \left( \frac{\gamma+1}{2} \right)}{\Gamma\left( \frac{\gamma}{2} \right)} \frac{1}{(1 + r^2 t^2)^\frac{\gamma +1}{2}}, \quad r,t>0
\end{equation}
which coincides for $\gamma=n$, $t=1/\sqrt{n}$ with the folded Student distribution. Related distributions are
\begin{align}
& Pr \left\lbrace  \frac{1}{1 + R^\gamma(T_t)} \in dr \right\rbrace  /dr \\
= & \frac{2t}{\sqrt{\pi}} \frac{\Gamma\left( \frac{\gamma +1}{2} \right)}{\Gamma\left( \frac{\gamma}{2} \right)} (1-w)^{\gamma -1} \left[ \frac{\frac{t}{1+t^2}}{\left( w - \frac{1}{1+t^2} \right)^2 + \frac{t^2}{(1+t^2)^2}} \right]^{\frac{\gamma +1}{2}}, \quad 0 < w < 1.  \nonumber
\end{align}
Furthermore, we obtain also that
\begin{equation}
Pr \left\lbrace  \frac{t^3}{t^2 + |R^\gamma(T_t)|^2} \in dr \right\rbrace /dr = \frac{1}{t}\frac{1}{B(\frac{\gamma}{2}, \frac{1}{2})} \left( \frac{r}{t} \right)^{\frac{1}{2} - 1} \left( 1- \frac{r}{t} \right)^{\frac{\gamma}{2} - 1},
\end{equation}
with $0 < r < t$ and $\gamma >0$. For $\gamma=1$ we have the arcsin law. Bessel processes $R^\gamma(T^\mu_t)$, $t>0$, $\gamma>0$ stopped at first passage times  $T_t=\inf\{ s:\, B^\mu(s) = t \}$ where $B^\mu$ is Brownian motion with drift $\mu$ are examined in Section 3. In particular we prove that
\begin{align}
q_\mu(r,t) = & Pr\left\lbrace R^\gamma(T_t^\mu) \in dr \right\rbrace /dr\nonumber \\
= & \frac{4t\, e^{t\mu} r^{\gamma-1}}{2^\frac{\gamma}{2} \Gamma\left( \frac{\gamma}{2} \right) \sqrt{2\pi}} \left(\frac{\mu^2}{r^2+t^2} \right)^\frac{\gamma +1}{4} K_\frac{\gamma + 1}{2} \left(\mu \sqrt{r^2+t^2}\right) 
\end{align}
with $r\geq 0$, $t>0$, $\gamma >0$, $\mu \geq 0$.

The last section is devoted to compositions involving the hyperbolic Brownian motion, that is a diffusion on the Poincar\'e upper half-space $H^n_+ = \{ x_1, \ldots , x_n :\, x_n >0 \}$ with particular attention to the planar case $H^2_+$ and the three-dimensional Poincar\'e half-space $H^3_+$. In the space $H^2_+$ we study the hyperbolic distance from the origin of a hyperbolic Brownian motion stopped at the first-passage time $T_t$, $t>0$ of the standard Brownian motion whose probability law can be explicitly written as
\begin{equation}
p_{J_2}(\eta, t) = \frac{\sinh \eta}{\pi \sqrt{2^3}} \int_\eta^\infty \frac{\varphi}{\sqrt{\cosh \varphi - \cosh \eta}} \frac{t}{t^2 + \varphi^2} K_2 \left( \frac{1}{2} \sqrt{t^2 + \varphi^2} \right) d\varphi
\label{hequ}
\end{equation}
with $\eta >0$, $t>0$. In $H^3_+$ the distribution of $J_3(t)$, $t>0$ reads
\begin{equation}
p_{J_3}(\eta , t) = \frac{2\sqrt{2}}{\pi} \frac{\eta \, t\, \sinh \eta }{(\eta^2 + t^2)} K_2 \left(\sqrt{\eta^2 + 2 t^2} \right), \quad \eta >0\, t>0.
\label{heqd}
\end{equation}
The evaluation of the integrals leading to \eqref{hequ} and \eqref{heqd} necessitates the following formula
\begin{equation}
\int_0^\infty x^{\nu -1} \exp \left\lbrace -\beta x^p - \gamma x^{-p} \right\rbrace dx = \frac{2}{p} \left( \frac{\gamma}{\beta} \right)^\frac{\nu}{2p} K_\frac{\nu}{p} \left(2 \sqrt{\gamma \beta} \right) \label{formula:K}
\end{equation}
where $p,\gamma,\beta,\nu >0$ and  $K_\nu$ is the modified Bessel function (see \cite[formula 3.478]{GR}). The equations governing \eqref{hequ} and \eqref{heqd} are respectively
\begin{align}
&-\frac{\partial^2 p_{J_2}}{\partial t^2} = \frac{\partial^2 p_{J_2}}{\partial \eta^2} - \frac{\partial}{\partial \eta} \left( \frac{1}{\tanh \eta} p_{J_2} \right) , \quad \eta>0, \, t>0\\
&-\frac{\partial^2 p_{J_3}}{\partial t^2} = \frac{\partial^2 p_{J_3}}{\partial \eta^2} - 2 \frac{\partial}{\partial \eta} \left( \frac{1}{\tanh \eta} p_{J_3} \right) , \quad \eta>0, \, t>0.
\end{align}
The hyperbolic distance of a hyperbolic Brownian motion plays in the non-Euclidean spaces $H^n_+$, $n=2,3,\ldots$ the same role of Bessel processes in the Euclidean spaces. The structure of the probability law of $B^{hp}_2(t)$, $t>0$ is rather complicated (see formulae \eqref{LAA} and \eqref{LAAA} below) and therefore we have restricted ourselves only to compositions involving first-passage times. Much more flexibility is allowed by three-dimensional hyperbolic Brownian motion $B^{hp}_3(t)$, $t>0$. Millson formula (see \cite{GN98}), in principle, permits us to examine compositions of higher-dimensional hyperbolic Brownian motions stopped at random times.

\section{Composition of Bessel Processes with different types of processes}
\label{seZ}
We first present  some information about the Bessel process of order $\gamma>0$ and starting from $x \geq 0$. The Bessel process $R^\gamma_x(t)$, $t>0$ is a diffusion with law
\begin{equation}
p(t,r;0,x) = \frac{r}{t} \left( \frac{r}{x} \right)^{\frac{\gamma}{2} -1} \exp\left\lbrace -\frac{x^2 + r^2}{2t}\right\rbrace I_{\frac{\gamma}{2} -1} \left( \frac{x r}{t} \right), \quad x, r\geq 0,\, t>0
\label{dist:besG}
\end{equation}
governed by the infinitesimal generator
\begin{equation}
\mathcal{A} = \frac{1}{2} \left\lbrace \frac{\partial^2}{\partial x^2} + \frac{\gamma - 1}{x} \frac{\partial}{\partial x} \right\rbrace .
\label{generatorUno}
\end{equation}
The Bessel function $I_\nu(z)$ is defined as
\begin{equation}
I_\nu(z) = \sum_{k=0}^{\infty} \frac{\left( z/2 \right)^{\nu + 2k}}{k!\, \Gamma(k + \nu +1)}.
\end{equation}
For $\gamma = n$ , $n \in \mathbb{N}$, the Bessel process represents the Euclidean distance of a $n$-dimensional Brownian motion $\left( B_1(t), \ldots , B_n(t) \right)$, $t>0$ from the origin. For 
\begin{equation}
R^n_x(t) = \sqrt{\sum_{j=1}^n B^2_j(t)}, \quad t>0
\end{equation}
the explicit law  \eqref{dist:besG} reads
\begin{equation}
p(t,r;0,x) = \frac{r}{t} \left( \frac{r}{x} \right)^{\frac{n}{2}-1} \exp\left\lbrace -\frac{x^2+r^2}{2t}\right\rbrace  I_{\frac{n}{2}-1} \left( \frac{x r}{t} \right), \quad x,r \geq 0, \; t>0
\end{equation}
and simplifies for $x=0$ as
\begin{align}
p(t,r;0,0) = & 2 \frac{r^{n-1} e^{-\frac{r^2}{2t}}}{\left( 2t\right)^\frac{n}{2} \Gamma\left(\frac{n}{2} \right)} = r^{n-1} k(r,t), \quad r\geq 0,\;t>0 \label{dist:k}
\end{align}
where the function $k=k(r,t)$, $r \in \mathbb{R}^+$, $t>0$ is the heat kernel satisfying the p.d.e.
\begin{equation}
\frac{\partial k}{\partial t} = \frac{1}{2} \frac{1}{r^{n-1}} \frac{\partial}{\partial r} \left\lbrace r^{n-1} \frac{\partial}{\partial r} \right\rbrace k = \mathcal{A}\, k
\end{equation}
where $\mathcal{A}$ is that in \eqref{generatorUno} with $\gamma = n$. We point out that the probability density \eqref{dist:k} satisfies the p.d.e.
\begin{align}
\frac{\partial p}{\partial t} = & \frac{1}{2} \left\lbrace \frac{\partial^2 p}{\partial r^2} - (\gamma -1) \frac{\partial}{\partial r} \left( \frac{p}{r} \right) \right\rbrace = \frac{1}{2} \left\lbrace \frac{\partial^2 p}{\partial r^2} -\frac{\gamma -1}{r} \frac{\partial p}{\partial r} + \frac{\gamma -1}{r^2}p \right\rbrace \label{derA}
\end{align}
for $\gamma = n$. The differential operator figuring in \eqref{derA} is usually referred to as the adjoint of \eqref{generatorUno}. For the operators introduced above the following interesting fact turns out to be very useful
\begin{equation}
\frac{\partial p}{\partial t} = \mathcal{A}^* p \quad \textrm{and} \quad \frac{\partial k}{\partial t} = \mathcal{A} k.
\label{OPProp}
\end{equation}
Our interest here is to study the composition of the Bessel process $_1R^\gamma_0(t)$, $t>0$ outlined above with different processes or also with an independent Bessel process $_2R^\gamma_0(t)$, $t>0$.

\subsection{The Iterated Bessel process}
We consider here the iterated Bessel process
\begin{equation}
I_R(t)=\,_1R^\gamma_0(\,_2R^\gamma_0(t)), \quad t>0
\label{proc:IR}
\end{equation} 
where $_1R^\gamma_0$, $_2R^\gamma_0$ are independent Bessel processes of dimension $\gamma$ starting at $x=0$ and possessing density 
\begin{equation}
p(r,t) = 2\frac{r^{\gamma -1} e^{-\frac{r^2}{2t}}}{(2t)^\frac{\gamma}{2} \Gamma\left(\frac{\gamma}{2}\right)}, \quad r \geq 0, \; t>0.
\end{equation}
The distribution of \eqref{proc:IR} reads
\begin{equation}
q(r,t) = 2^2\int_0^\infty \frac{r^{\gamma -1} e^{-\frac{r^2}{2s}}}{(2s)^\frac{\gamma}{2} \Gamma\left(\frac{\gamma}{2}\right)} \frac{s^{\gamma -1} e^{-\frac{s^2}{2t}}}{(2t)^\frac{\gamma}{2} \Gamma\left(\frac{\gamma}{2}\right)} ds, \quad r \geq 0,\; t>0
\label{dist:exrendsthe}
\end{equation}
and has Mellin transform equal to
\begin{equation}
\mathcal{M}\left\lbrace q(\cdot , t) \right\rbrace (\eta) =  (2^3 t)^\frac{\eta -1}{4} \frac{\Gamma\left(\frac{\eta}{2} + \frac{\gamma}{2} - \frac{1}{2} \right) \Gamma\left( \frac{\eta}{4} + \frac{\gamma}{2} - \frac{1}{4} \right)}{\Gamma^2\left( \frac{\gamma}{2} \right)}\label{mellJJ}, \quad \Re\{ \eta \} > 1 - \gamma.
\end{equation}
The distribution \eqref{dist:exrendsthe} can be expressed in terms of Fox functions which are defined in the following manner:
\begin{equation}
H^{m,n}_{p,q}(x) = H^{m,n}_{p,q}\left[ x \bigg| \begin{array}{l} (a_i, \alpha_i)_{i=1, .. , p}\\ (b_j, \beta_j)_{j=1, .. , q}  \end{array} \right] = \frac{1}{2\pi i} \int_{\theta - i \infty}^{\theta + i \infty} \mathcal{M}^{m,n}_{p,q} (\eta) x^{-\eta} d\eta
\end{equation}
where $\theta \in \mathbb{R}$  and
\begin{equation}
\mathcal{M}^{m,n}_{p,q} (\eta)  = \frac{\prod_{j=1}^{m} \Gamma(b_j + \eta \beta_j) \prod_{i=1}^{n} \Gamma(1-a_i - \eta \alpha_i)}{\prod_{j=m+1}^{q} \Gamma(1-b_j - \eta \beta_j) \prod_{i=n+1}^p \Gamma(a_i + \eta \alpha_i)}.
\label{mellinHfox}
\end{equation}
By direct inspection of \eqref{mellJJ} we see that in our case $n=0$, $m=2$, $p=q=2$, $b_1=\gamma/2-1/2$, $b_2=\gamma/2-1/4$, $\beta_1=1/2$, $\beta_2=1/4$, $a_1=a_2=\gamma/2$, $\alpha_1=\alpha_2=0$. By considering the following property of the Mellin transform
\begin{equation}
\int_0^\infty x^{\eta -1} f(ax) dx = a^{-\eta} \int_0^\infty x^{\eta -1} f(x) dx
\end{equation}
we conclude that \eqref{dist:exrendsthe} can be represented in terms of Fox functions as
\begin{equation}
q(r,t) = \frac{1}{(2^3 t)^{1/4}} H^{2,0}_{2,2}\left[ \frac{r}{(2^3 t)^{1/4}} \bigg| \begin{array}{cc} (\frac{\gamma}{2}, 0); & (\frac{\gamma}{2}, 0) \\ (\frac{\gamma}{2}-\frac{1}{2} , \frac{1}{2}); & (\frac{\gamma}{2} - \frac{1}{4}, \frac{1}{4})  \end{array} \right].
\label{asdfgh}
\end{equation}
In view of the property of Fox functions (see \cite{MS73})
\begin{equation}
H^{m,n}_{p,q}(x) = \frac{1}{x^c} \; H^{m,n}_{p,q}\left[ x \bigg| \begin{array}{l} (a_i + c \alpha_i, \alpha_i)_{i=1, .. , p}\\ (b_j + c \beta_j, \beta_j)_{j=1, .. , q}  \end{array} \right], \quad c \in \mathbb{R}. \label{propH2}
\end{equation}
for $c=1$,  \eqref{asdfgh} can be written as
\begin{equation}
q(r,t) = \frac{1}{r \, (2^3 t)^{1/4}} H^{2,0}_{2,2}\left[ \frac{r}{(2^3 t)^{1/4}} \bigg| \begin{array}{cc} (\frac{\gamma}{2}, 0); & (\frac{\gamma}{2}, 0) \\ (\frac{\gamma}{2} , \frac{1}{2}); & (\frac{\gamma}{2}, \frac{1}{4})  \end{array} \right].
\end{equation}
Moreover, the Mellin transform \eqref{mellJJ} generalizes the moments of \eqref{proc:IR} and after some easy calculation, we have
\begin{equation}
E \left\lbrace I_R(t) \right\rbrace^m = 2^\frac{m}{2} \frac{\Gamma\left( \frac{m+n}{2} \right) \Gamma\left( \frac{m}{4} + \frac{n}{2} \right)}{\Gamma^2\left( \frac{n}{2} \right)} \left( 2t \right)^\frac{m}{4} , \quad m=1,2,\ldots
\label{momentsIR}
\end{equation}
Formula \eqref{momentsIR} shows that the mean distance of the iterated Bessel process increases as $t^\frac{1}{4}$ and the dimension $\gamma$ of the space where the iterated Bessel processes develop enters only in the multiplying coefficient in \eqref{momentsIR}. In the next theorem we derive the governing equation of the law \eqref{dist:exrendsthe} of the process \eqref{proc:IR}.
\begin{te}
The density function \eqref{dist:exrendsthe} satisfies, for $0 < \gamma \leq 1$ the equation
\begin{align}
\frac{\partial q}{\partial t} = & \frac{1}{2^3} \left( \frac{\partial^2}{\partial r^2} - (\gamma -1) \frac{\partial}{\partial r}\frac{1}{r} \right) \left( \frac{\partial^2}{\partial r^2} + \frac{\gamma -1}{r} \frac{\partial}{\partial r} + \frac{(\gamma -1)(3-2\gamma)}{r^2}\right) q \nonumber \\
+ &  \frac{1}{2(2t)^\frac{\gamma}{2} \Gamma \left(  \frac{\gamma}{2}\right)}\left( \frac{\partial^2}{\partial r^2} - (\gamma -1) \frac{\partial}{\partial r}\frac{1}{r} \right) \delta(r). \label{pde:eqUDD}
\end{align}
For $\gamma >1$ the governing equation becomes
\begin{equation}
\frac{\partial q}{\partial t} = \frac{1}{2^3} \left( \frac{\partial^2}{\partial r^2} - (\gamma -1) \frac{\partial}{\partial r}\frac{1}{r} \right) \left( \frac{\partial^2}{\partial r^2} + \frac{\gamma -1}{r} \frac{\partial}{\partial r} + \frac{(\gamma -1)(3-2\gamma)}{r^2}\right) q.
\label{pde:eqUDD2}
\end{equation}
\label{te:POP}
\end{te}
\begin{proof}
It is convenient to write the distribution \eqref{dist:exrendsthe} as
\begin{equation}
q(r,t) = \int_0^\infty p(r,s) p(s,t) ds
\label{dist:IRpp}
\end{equation}
where
\begin{equation*}
p(r,t) = 2\frac{r^{\gamma -1} e^{-\frac{r^2}{2t}}}{(2t)^\frac{\gamma}{2} \Gamma\left( \frac{\gamma}{2} \right)}= r^{\gamma -1} k(r,t), \quad r\geq 0,\; t>0,\; \gamma>0.
\end{equation*}
We start our proof by evaluating the time-derivative of \eqref{dist:IRpp} as
\begin{align*}
\frac{\partial q}{\partial t} = & \int_0^\infty p(r,s) \frac{\partial}{\partial t} p(s,t) ds = \left[ \textrm{in view of } \eqref{derA} \right] \\
= &\int_0^\infty p(r,s) \frac{1}{2} \left\lbrace \frac{\partial^2}{\partial s^2} - (\gamma -1) \frac{\partial}{\partial s} \frac{1}{s} \right\rbrace p(s,t) ds\\
= & \frac{1}{2} p(r,s) \left\lbrace \frac{\partial}{\partial s} p(s,t) - \frac{\gamma -1}{s} p(s,t)  \right\rbrace \Bigg|_{s=0}^{s=\infty}\\
- &\frac{1}{2} \int_0^{\infty} \frac{\partial}{\partial s} p(r,s) \left\lbrace \frac{\partial}{\partial s} p(s,t) - \frac{(\gamma -1)}{s} p(s,t) \right\rbrace ds
\end{align*}
We note that
\begin{equation}
\frac{1}{2} \left\lbrace \frac{\partial}{\partial s} p(s,t) - \frac{\gamma -1}{s}p(s,t) \right\rbrace p(r,s) \Bigg|_{s=0}^{s=\infty} = \delta(r) \lim_{s \to 0^+} \frac{s^\gamma e^{-\frac{s^2}{2t}}}{t (2t)^\frac{\gamma}{2} \Gamma\left( \frac{\gamma}{2} \right)} 
\end{equation}
and thus for $s \to 0^+$ and $s \to \infty$ vanishes for all $\gamma >0$. A further integration by parts yields that
\begin{align*}
\int_0^\infty \frac{\partial}{\partial s}p(r,s) \frac{\partial}{\partial s} p(s,t) ds = & \frac{\partial}{\partial s} p(r,s) p(s,t) \Bigg|_{s=0}^{s=\infty} - \int_0^\infty \frac{\partial^2}{\partial s^2} p(r,s) p(s,t) ds.
\end{align*}
We note that
\begin{align*}
\frac{\partial}{\partial s}  p(r,s) & p(s,t) \Bigg|_{s=0}^{s=\infty} = r^{\gamma -1} \frac{\partial}{\partial s} k(r,s) p(s,t) \Bigg|_{s=0}^{s=\infty}\\
= & \left[ \textrm{by } \eqref{OPProp} \right] = - \frac{r^{\gamma -1}}{2} \left( \frac{\partial^2}{\partial r^2} + \frac{\gamma -1}{r} \frac{\partial}{\partial r} \right) \frac{\delta(r)}{r^{\gamma -1}} \frac{2}{(2t)^\frac{\gamma}{2} \Gamma\left(\frac{\gamma}{2}\right)}
\end{align*}
for $0 < \gamma \leq 1$, because
\begin{equation}
\frac{\partial}{\partial s}p(r,s) = \frac{r^{\gamma -1}}{2} \left( \frac{\partial^2}{\partial r^2} + \frac{\gamma -1}{r} \frac{\partial}{\partial r} \right) \frac{p(r,s)}{r^{\gamma -1}}
\end{equation}
and $p(r,s) \to \delta(r)$ for $s \to 0^+$. Furthermore,
\begin{equation*}
p(s,t) = const \, s^{\gamma -1} e^{-\frac{s^2}{2t}}
\end{equation*}
and this explains why $\gamma \leq 1$ implies the appearance of the delta function. By collecting all pieces together we have that 
\begin{align}
\frac{\partial q}{\partial t} = & \frac{r^{\gamma -1}}{2} \left( \frac{\partial^2}{\partial r^2} + \frac{\gamma -1}{r} \frac{\partial}{\partial r} \right) \frac{\delta(r)}{r^{\gamma -1}} \frac{1}{(2t)^\frac{\gamma}{2} \Gamma\left( \frac{\gamma}{2} \right)} \label{eq220} \\
+ & \frac{1}{2}\int_0^\infty \left\lbrace \frac{\partial^2}{\partial s^2} p(r,s)- \frac{\gamma-1}{s} \frac{\partial}{\partial s} p(r,s) \right\rbrace p(s,t) ds.\nonumber
\end{align}
By exploiting the following manipulations
\begin{align*}
- \frac{1}{s} \frac{\partial}{\partial s} p(r,s) = & -\frac{r^{\gamma -1}}{s} \frac{\partial}{\partial s} k(r,s) \\
= & \frac{r^{\gamma -1}}{2} \left( \frac{\partial^2}{\partial r^2} + \frac{\gamma -1}{r} \frac{\partial}{\partial r} \right) \frac{1}{r} \left( -\frac{r}{s} \right) k(r,s) \\
= & \frac{r^{\gamma -1}}{2} \left( \frac{\partial^2}{\partial r^2} + \frac{\gamma -1}{r} \frac{\partial}{\partial r} \right) \frac{1}{r} \frac{\partial}{\partial r} \frac{p(r,s)}{r^{\gamma -1}}
\end{align*}
equation \eqref{eq220} takes the form
\begin{align}
\frac{\partial q}{\partial t} = & \frac{r^{\gamma -1}}{2^3} \left( \frac{\partial^2}{\partial r^2} + \frac{\gamma -1}{r} \right)^2 \frac{q}{r^{\gamma -1}} \label{eq221}\\
+ &\frac{(\gamma -1)}{2^2} r^{\gamma -1} \left( \frac{\partial^2}{\partial r^2} + \frac{\gamma -1}{r} \frac{\partial}{\partial r} \right) \frac{1}{r} \frac{\partial}{\partial r} \frac{q}{r^{\gamma -1}}\\
+ & \frac{1}{2} \frac{r^{\gamma -1}}{(2t)^\frac{\gamma}{2} \Gamma\left(\frac{\gamma}{2} \right)} \left( \frac{\partial^2}{\partial r^2} + \frac{\gamma -1}{r} \frac{\partial}{\partial r} \right) \frac{\delta(r)}{r^{\gamma -1}} \nonumber
\end{align}
In order to transform equation \eqref{eq221} into the form of the statement of Theorem \ref{te:POP} we note that
\begin{equation}
\left( \frac{\partial^2}{\partial r^2} + \frac{\gamma -1}{r} \frac{\partial}{\partial r} \right) \frac{q}{r^{\gamma -1}} = \frac{1}{r^{\gamma -1}} \left( \frac{\partial^2}{\partial r^2} - (\gamma -1) \frac{\partial}{\partial r} \frac{1}{r} \right) q
\label{opProp}
\end{equation}
as a direct check shows. Furthermore, if we write
\begin{equation*}
\left( \frac{\partial^2}{\partial r^2} - (\gamma -1) \frac{\partial}{\partial r} \frac{1}{r} \right) q(r,t) = w(r,t)
\end{equation*}
and apply again \eqref{opProp} we arrive at the following explicit expression
\begin{equation*}
\left( \frac{\partial^2}{\partial r^2} + \frac{\gamma -1}{r} \frac{\partial}{\partial r} \right)^2 \frac{q}{r^{\gamma -1}} = \frac{1}{r^{\gamma -1}} \left( \frac{\partial^2}{\partial r^2} - (\gamma -1) \frac{\partial}{\partial r} \frac{1}{r} \right)^2 q.
\end{equation*}
By applying the same trick as above we have that
\begin{align*}
& \left( \frac{\partial^2}{\partial r^2} + \frac{\gamma -1}{r} \frac{\partial}{\partial r} \right) \frac{1}{r} \frac{\partial}{\partial r} \frac{q}{r^{\gamma -1}} \\
= & \left( \frac{\partial^2}{\partial r^2} + \frac{\gamma -1}{r} \frac{\partial}{\partial r} \right) \frac{1}{r^{\gamma -1}} \left\lbrace - \frac{(\gamma -1)}{r^2} + \frac{1}{r} \frac{\partial}{\partial r} \right\rbrace q\\
= & \frac{1}{r^{\gamma -1}} \left( \frac{\partial^2}{\partial r^2} - (\gamma -1) \frac{\partial}{\partial r} \frac{1}{r} \right) \left\lbrace - \frac{(\gamma -1)}{r^2} + \frac{1}{r} \frac{\partial}{\partial r} \right\rbrace q.
\end{align*}
An analogous step must be applied to the singular term involving the Dirac delta function so that result \eqref{pde:eqUDD} follows. 
\end{proof}

\begin{os}
\normalfont
For $\gamma=1$ equation \eqref{pde:eqUDD} becomes
\begin{equation*}
\frac{\partial q}{\partial t} = \frac{1}{2^3} \frac{\partial^4 q}{\partial r^4} + \frac{1}{2 \sqrt{2 \pi t}} \frac{d^2 \delta(r)}{d r^2} 
\end{equation*}
which is the fourth-order equation governing the iterated Brownian motion. This is because the process $I_R(t)$, $t>0$ for $\gamma=1$ becomes the reflected iterated Brownian motion. It is well-known that the law of the iterated Brownian motion satisfies the fractional equation
\begin{equation*}
\frac{\partial^\frac{1}{2} q}{\partial t^\frac{1}{2}} = \frac{1}{2^\frac{3}{2}} \frac{\partial^2 q}{\partial x^2}.
\end{equation*}
\end{os}

\begin{os}
\normalfont
For $\gamma=\frac{3}{2}$, equation \eqref{pde:eqUDD2} takes the form
\begin{align*}
\frac{\partial q}{\partial t} = & \frac{1}{2^3} \left( \frac{\partial^2}{\partial r^2} - \frac{1}{2} \frac{\partial}{\partial r} \frac{1}{r} \right) \left( \frac{\partial^2}{\partial r^2} + \frac{1}{2r} \frac{\partial}{\partial r} \right) q\\
= & \frac{1}{2^3} \left( \frac{\partial^4}{\partial r^4} - \frac{3}{2^2 r^2} \frac{\partial^2}{\partial r^2} + \frac{3}{2 r^3} \frac{\partial}{\partial r} \right)q =  \frac{1}{2^3} \left( \frac{\partial^4}{\partial r^4} - \frac{3}{2^2} \frac{\partial}{\partial r} \left( \frac{1}{r^2} \frac{\partial}{\partial r} \right) \right) q.
\end{align*}
This is the simplest equation involving iterated Bessel functions.
\end{os}

If we consider the slightly modified process
\begin{equation}
J_R(t) = \,_1R^\gamma_0 (|\, _2R^\gamma_0(t)|^2), \quad t>0
\end{equation}
we have the advantage that the law can be expressed explicitly in terms of modified Bessel functions. In fact, we have that the probability density $g$ of $J_R(t)$, $t>0$ reads
\begin{equation}
g(r,t) = \frac{2\, r^{\gamma -1}}{(4t)^\frac{\gamma}{2} \Gamma^2 \left( \frac{\gamma}{2} \right)} \int_0^\infty \frac{1}{s} e^{-\frac{r^2}{2s} - \frac{s}{2t}} ds = \frac{2^2\, r^{\gamma -1}}{(4t)^\frac{\gamma}{2} \Gamma^2 \left( \frac{\gamma}{2} \right)} K_0 \left( \frac{r}{\sqrt{t}} \right)
\label{dist:g}
\end{equation}
where
\begin{equation*}
K_0(x) = \int_0^\infty \frac{1}{s} \exp\left\lbrace -\frac{x^2}{4s^2} - s^2\right\rbrace  ds, \quad x>0
\end{equation*}
(see \cite[formula 3.478]{GR}).
\begin{te}
The probability law of the process $J_R(t)$, $t>0$ is governed by the following third-order equation
\begin{equation}
\frac{\partial q}{\partial t} = -\frac{1}{2} \left\lbrace r \frac{\partial^3 q}{\partial r^3} + 2 (2 - \gamma) \frac{\partial^2 q}{\partial r^2} + \frac{\partial q}{\partial r} \frac{(\gamma -1)^2}{r}  \right\rbrace, \quad r,t>0 .
\label{eqBB2}
\end{equation}
\end{te}
\begin{proof}
We start by working on the probability density \eqref{dist:g} which can be written as
\begin{equation}
q(r,t) = \int_0^\infty p(r,s) l(s,t) ds, \quad r \geq 0, \: t>0
\end{equation}
with
\begin{equation}
p(r,s) = 2\frac{r^{\gamma -1} e^{-\frac{r^2}{2s}}}{(2s)^\frac{\gamma}{2} \Gamma\left( \frac{\gamma}{2} \right)}, \quad r \geq 0,\, s>0
\end{equation}
and
\begin{equation}
l(s,t) = \frac{s^{\frac{\gamma}{2}-1} e^{-\frac{s}{2t}}}{(2t)^\frac{\gamma}{2} \Gamma\left( \frac{\gamma}{2} \right)}, \quad s>0, \, t>0
\label{density:l}
\end{equation}
We note that the density \eqref{density:l} is a solution of the following equation
\begin{equation}
\frac{\partial l}{\partial t} = 2s \frac{\partial^2 l}{\partial s^2} - (\gamma -4) \frac{\partial l}{\partial s}
\label{eq:densityl}
\end{equation}
as a direct check shows (see also \cite{RoiY08}). In view of \eqref{eq:densityl} from \eqref{density:l} we have that
\begin{align*}
\frac{\partial q}{\partial t} =& \int_{0}^\infty p(r,s) \left[ 2s \frac{\partial^2 l}{\partial s^2} - (\gamma -4) \frac{\partial l}{\partial s} \right] ds\\
= &- 2 \int_0^\infty s\, \frac{\partial}{\partial s} p(r,s) \frac{\partial}{\partial s}l(s,t) ds - (\gamma -2) \int_0^\infty p(r,s)  \frac{\partial}{\partial s} l(s,t)  ds\\
- & \left( \gamma - 2 \right) p(r,0) \frac{1}{(2t)^\frac{\gamma}{2} \Gamma\left( \frac{\gamma}{2}\right)} \mathbb{I}_{(\gamma \leq 2)}\\
= & - 2 \int_0^\infty s\, \frac{\partial}{\partial s} p(r,s) \frac{\partial}{\partial s}l(s,t) ds + (\gamma -2) \int_0^\infty \frac{\partial}{\partial s} p(r,s)  l(s,t)  ds\\
= & 2 \int_0^\infty s\, \frac{\partial^2}{\partial s^2} p(r,s) l(s,t) ds + \gamma \int_0^\infty \frac{\partial}{\partial s} p(r,s)  l(s,t)  ds\\
+ & 2\frac{s^\frac{\gamma}{2}}{(2t)^\frac{\gamma}{2} \Gamma\left( \frac{\gamma}{2} \right)} \frac{\partial}{\partial s}p(r,s) \Bigg|_{s=0}\quad (\gamma >0)\\
= & 2 r^{\gamma -1 } \frac{1}{2^2} \left( \frac{\partial^2}{\partial r^2} + \frac{\gamma -1}{r}\frac{\partial}{\partial r} \right)^2 \int_0^\infty s\, k(r,s) l(s,t) ds \\
+ & \frac{\gamma}{2} \left( \frac{\partial^2}{\partial r^2} - (\gamma -1)\frac{\partial}{\partial r} \frac{1}{r} \right) q(r,t).
\end{align*}
We observe that
\begin{align*}
& \left( \frac{\partial^2}{\partial r^2} + \frac{\gamma -1}{r}\frac{\partial}{\partial r} \right)^2 \int_0^\infty s\, k(r,s) l(s,t) ds\\
= & \left( \frac{\partial^2}{\partial r^2} + \frac{\gamma -1}{r}\frac{\partial}{\partial r} \right) \left( \frac{\partial}{\partial r} + \frac{\gamma -1}{r} \right) \int_0^\infty s\,\left(-\frac{r}{s} \right) k(r,s) l(s,t) ds\\
= & - \left( \frac{\partial^2}{\partial r^2} + \frac{\gamma -1}{r}\frac{\partial}{\partial r} \right) \left( \frac{\partial}{\partial r} + \frac{\gamma -1}{r} \right) \frac{1}{r^{\gamma -2}}  q(r,t)
\end{align*}
and
\begin{equation*}
\left( \frac{\partial}{\partial r} + \frac{\gamma -1}{r} \right) \frac{1}{r^{\gamma -2}}  q(r,t) = \frac{1}{r^{\gamma -1}} \left( 1+ r\frac{\partial}{\partial r}  \right) q(r,t).
\end{equation*}
Thus, we obtain 
\begin{align*}
\frac{\partial q}{\partial t} = & - \frac{r^{\gamma -1 }}{2} \left( \frac{\partial^2}{\partial r^2} + \frac{\gamma -1}{r} \frac{\partial}{\partial r} \right) \frac{1}{r^{\gamma -1}} \left( 1 + r \frac{\partial}{\partial r} \right) q(r,t) \\
+ & \frac{\gamma}{2} \left( \frac{\partial^2}{\partial r^2} - (\gamma -1)\frac{\partial}{\partial r} \frac{1}{r} \right) q(r,t)\\
= & \left[ \textrm{ by } \eqref{opProp} \right] = - \frac{1}{2} \left( \frac{\partial^2}{\partial r^2} - (\gamma -1)\frac{\partial}{\partial r} \frac{1}{r} \right) \left( 1 + r \frac{\partial}{\partial r} \right) q(r,t) \\
+ & \frac{\gamma}{2} \left( \frac{\partial^2}{\partial r^2} - (\gamma -1)\frac{\partial}{\partial r} \frac{1}{r} \right) q(r,t)\\
= & - \frac{1}{2} \left( \frac{\partial^2}{\partial r^2} - \frac{(\gamma -1)}{r}\frac{\partial}{\partial r} + \frac{(\gamma -1)}{r^2} \right) \left[ r \frac{\partial}{\partial r} - (\gamma -1) \right] q(r,t)
\end{align*}
Formula \eqref{eqBB2} can be derived by simple calculation. 
\end{proof}

\subsection{The Bessel process at first-passage times}
\label{d}
Let $T_t=\inf\{s: \,B(s)=t\}$ where $B$ is a Brownian motion (possibly with drift $\mu$) independent from the Bessel process $R^\gamma(t)$, $t>0$ starting from zero. In this section we study the new process $R^\gamma(T_t)$, $t>0$ concentrating our attention on its law and some related distributions. Stopping the Bessel process $R^\gamma$ at the random time $T_t$ can cause either a slowing down (with respect to the natural time) or a speed up of the time flow. The probability of slowing down is measured by the following integral
\begin{equation}
Pr\{ T_t \leq t \}=\int_0^t \frac{t\,e^{-\frac{t^2}{2x}}}{\sqrt{2\pi x^3}} dx
\end{equation} 
which decreases for all $t$ since 
\begin{equation}
\frac{d}{dt}Pr\{ T_t \leq t \}=-\frac{e^{-\frac{t}{2}}}{\sqrt{2\pi}}, \quad t>0.
\end{equation} 
Furthermore, we observe that
\begin{equation}
Pr\{ T_t \leq t\} = \sqrt{\frac{2}{\pi t}} e^{-\frac{t}{2}} - \int_0^t \frac{e^{-\frac{t^2}{2x}}}{t \sqrt{2\pi x}} dx \leq  \sqrt{\frac{2}{\pi t}} e^{-\frac{t}{2}} 
\end{equation}
and this confirms the asymptotic speed up of the time flow implied by the subordinator $T_t$, $t>0$.
We have now the explicit distribution of $R^\gamma(T_t)$, $t>0$.
\begin{te}
The distribution of $R^\gamma(T_t)$, $t>0$ reads
\begin{equation}
q(r,t) = P\{ R^\gamma (T_t) \in dr \} /dr = 2 \frac{\Gamma\left( \frac{\gamma +1}{2} \right)}{\sqrt{\pi}\, \Gamma\left(\frac{\gamma}{2}\right)} \frac{t r^{\gamma -1}}{\left( r^2 + t^2 \right)^\frac{\gamma+1}{2}} , \quad r,t>0. \label{dist:coincides}
\end{equation}
\end{te}
\begin{proof}
\begin{align}
q(r,t) = &2 \int_0^\infty \frac{r^{\gamma -1} e^{-\frac{r^2}{2s}}}{(2s)^\frac{\gamma}{2} \Gamma\left( \frac{\gamma}{2} \right)}  \frac{t e^{-\frac{t^2}{2s}}}{\sqrt{2 \pi s^3}} ds \nonumber \\
= & \frac{2t r^{\gamma -1}}{2^\frac{\gamma +1}{2} \sqrt{\pi} \Gamma\left( \frac{\gamma}{2} \right)} \int_0^\infty s^{-\frac{\gamma +3}{2}} e^{-\frac{1}{2s}(r^2+t^2)} ds = \left[ s=\frac{1}{w}\left( \frac{r^2 + t^2}{2} \right) \right] \nonumber \\
= & \frac{2t r^{\gamma -1}}{2^\frac{\gamma +1}{2} \sqrt{\pi} \Gamma\left( \frac{\gamma}{2} \right)} \left( \frac{r^2 + t^2}{2} \right)^{-\frac{\gamma +1}{2}} \int_0^\infty e^{-w} w^{\frac{\gamma + 1}{2} - 1 } dw \nonumber \\
= & \frac{2t r^{\gamma -1}}{\sqrt{\pi}} \frac{\Gamma\left( \frac{\gamma+1}{2} \right)}{\Gamma \left( \frac{\gamma}{2} \right)} \frac{1}{\left( r^2 + t^2 \right)^\frac{\gamma+1}{2}}, \quad r,t>0. \nonumber
\end{align}
\end{proof}
\begin{os}
\normalfont
We check that the distribution \eqref{dist:coincides} integrates to unity,
\begin{align*}
\int_0^\infty Pr & \{ R^\gamma(T_t) \in dr \} = \frac{2t}{\sqrt{\pi}} \frac{\Gamma\left( \frac{\gamma +1}{2} \right)}{\Gamma\left(\frac{\gamma}{2} \right)} \int_0^\infty \frac{r^{\gamma -1}}{(r^2 + t^2)^\frac{\gamma +1}{2}} dr\\
= & \left[ r=t\sqrt{y} \right] = \frac{1}{\sqrt{\pi}} \frac{\Gamma\left( \frac{\gamma +1}{2} \right)}{\Gamma\left(\frac{\gamma}{2} \right)} \int_0^\infty \frac{y^{\frac{\gamma}{2} -1}}{(1+y)^\frac{\gamma +1}{2}} dy\\
= & \left[ \frac{y}{1+y} = w \right] = \frac{1}{\sqrt{\pi}} \frac{\Gamma\left( \frac{\gamma +1}{2} \right)}{\Gamma\left(\frac{\gamma}{2} \right)} \int_0^1 w^\frac{\gamma +1}{2} \left( \frac{w}{1-w} \right)^{-\frac{3}{2}} \frac{dw}{(1-w)^2}\\
= & \frac{1}{\sqrt{\pi}} \frac{\Gamma\left( \frac{\gamma +1}{2} \right)}{\Gamma\left(\frac{\gamma}{2} \right)} \int_0^\infty w^{\frac{\gamma}{2}-1} (1-w)^{\frac{1}{2} -1} dw=1 .
\end{align*}
With similar calculations we can obtain the $\mu$-moments for $0<\mu <1$
\begin{align*}
\int_0^\infty & r^\mu Pr \{ R^\gamma(T_t) \in dr \} =  \frac{t^\mu}{\sqrt{\pi}} \frac{\Gamma\left( \frac{\gamma +1}{2} \right)}{\Gamma\left( \frac{\gamma}{2} \right)} \int_0^\infty \frac{y^{\frac{\gamma +\mu}{2}-1}}{(1+y)^\frac{\gamma +1}{2}} dy\\
= & \frac{t^\mu}{\sqrt{\pi}} \frac{\Gamma\left( \frac{\gamma +1}{2} \right)}{\Gamma\left( \frac{\gamma}{2} \right)} \int_0^1 w^\frac{\gamma +1}{2} \left( \frac{w}{1-w} \right)^{\frac{\mu - 3}{2}} \frac{dw}{(1-w)^2} \\
= & \frac{t^\mu}{\sqrt{\pi}} \frac{\Gamma\left( \frac{\gamma +1}{2} \right)}{\Gamma\left( \frac{\gamma}{2} \right)} \int_0^1 w^{\frac{\gamma + \mu}{2} - 1} (1-w)^{\frac{1-\mu}{2} -1} dw = \frac{\Gamma\left( \frac{\gamma + \mu}{2}\right) \Gamma\left( \frac{1-\mu}{2} \right)}{\sqrt{\pi} \Gamma\left( \frac{\gamma}{2} \right)} t^\mu.
\end{align*}
We note also that for $\gamma=1$, the density law \eqref{dist:coincides} coincides with a folded Cauchy with scale parameter $t$ and location parameter equal to zero. 
\end{os}

For the distribution function of $R^\gamma(T_t)$, $t>0$ we have the following result
\begin{align}
Pr\{ & R^\gamma(T_t) > r\}= \frac{2t}{\sqrt{\pi}} \frac{\Gamma\left( \frac{\gamma +1}{2} \right)}{\Gamma\left( \frac{\gamma}{2} \right)} \int_r^\infty \frac{y^{\gamma -1}}{(t^2 + y^2)^\frac{\gamma + 1}{2}} dy\label{rtfv} \\
& = \frac{2t}{\sqrt{\pi}} \frac{\Gamma\left( \frac{\gamma +1}{2} \right)}{\Gamma\left( \frac{\gamma}{2} \right)} \left[ \frac{r^{\gamma -2}}{(\gamma -1)(t^2 + r^2)^\frac{\gamma -1}{2}} + \frac{\gamma -2}{\gamma -1} \int_r^\infty \frac{y^{\gamma -3}}{(t^2 + y^2)^\frac{\gamma -1}{2}}dy \right] \nonumber
\end{align}
for $\gamma >1$. The recursive formula \eqref{rtfv} yields some interesting special cases
\begin{equation}
Pr\{ R^\gamma(T_t) >r \} = \left\lbrace \begin{array}{ll} 
\frac{t}{(t^2 + r^2)^{1/2}}, & \gamma=2,\\
\frac{4t}{\pi} \left[ \frac{r}{t^2+r^2} + \frac{1}{2t}\left(\frac{\pi}{2} - \arctan \frac{r}{t} \right) \right], & \gamma =3,\\
\frac{t}{(t^2 + r^2)^{1/2}} + \frac{t\, r^2}{2(t^2 + r^2)^{3/2}}, & \gamma=4.
\end{array} \right .
\end{equation}

An interesting related distribution is presented in the next theorem.
\begin{te}
\label{prTe}
The process
\begin{equation}
\hat{R}^\gamma(t) = \frac{1}{1+R^\gamma(T_t)}, \quad t>0
\end{equation}
has distribution
\begin{align}
& Pr\{ \hat{R}^\gamma(t) \in dw \} / dw \nonumber \\
= & \frac{2t}{\sqrt{\pi}} \frac{\Gamma\left( \frac{\gamma +1}{2} \right)}{\Gamma\left( \frac{\gamma}{2} \right)} (1-w)^{\gamma -1} \left[ \frac{\frac{t}{1+t^2}}{\left( w - \frac{1}{1+t^2} \right)^2 + \frac{t^2}{(1+t^2)^2}} \right]^{\frac{\gamma +1}{2}}, \quad 0 < w < 1. \label{distr:qwA}
\end{align}
\end{te}
\begin{proof}
\begin{align}
& Pr\{ \hat{R}^\gamma(t) \in dw \} / dw \nonumber \\
= &\frac{d}{dw} \int_\frac{1-w}{w}^\infty Pr\{ R^\gamma(T_t) \in dr \}\\
= & \frac{2t}{\sqrt{\pi}} \frac{\Gamma\left( \frac{\gamma +1}{2} \right)}{\Gamma\left( \frac{\gamma}{2} \right)} \frac{1}{w^2} \left( \frac{1-w}{w} \right)^{\gamma -1} \left[ t^2 + \left( \frac{1-w}{w} \right)^2 \right]^{-\frac{\gamma +1}{2}}\nonumber \\
= & \frac{2t}{\sqrt{\pi}} \frac{\Gamma\left( \frac{\gamma +1}{2} \right)}{\Gamma\left( \frac{\gamma}{2} \right)} (1-w)^{\gamma -1} \left[ \left( w\sqrt{1+t^2} - \frac{1}{\sqrt{1+t^2}} \right)^2 + \frac{t^2}{(1+t^2)^2} \right]^{-\frac{\gamma +1}{2}} \nonumber \\
= & \frac{2t}{\sqrt{\pi}} \frac{\Gamma\left( \frac{\gamma +1}{2} \right)}{\Gamma\left( \frac{\gamma}{2} \right)} (1-w)^{\gamma -1} \left[ \frac{\frac{t}{1+t^2}}{\left( w - \frac{1}{1+t^2} \right)^2 + \frac{t^2}{(1+t^2)^2}} \right]^{\frac{\gamma +1}{2}}, \quad 0 < w < 1. \nonumber
\end{align}
\end{proof}

\begin{os}
\normalfont
For $\gamma=1$ the distribution \eqref{distr:qwA} offers the following expression
\begin{align}
Pr\{ \hat{R}^1(t) \in dw \} = & \frac{2}{\pi} \frac{\frac{t}{1+t^2}}{\left( w - \frac{1}{1+t^2} \right)^2 + \frac{t^2}{1+t^2}} = \frac{2}{\pi} \frac{A}{(w-B)^2 + A^2}. \label{LKwert}
\end{align}
We are able to check that \eqref{LKwert} integrates to unity. Indeed
\begin{align*}
\frac{2}{\pi} \int_0^1 \frac{A}{(w-B)^2 + A^2} dw = & \frac{2}{\pi} \int_{-\frac{B}{A}}^{\frac{1-B}{A}} \frac{dy}{1+y^2}\\
= & \frac{2}{\pi} \left\lbrace \arctan \frac{1-B}{A} + \arctan \frac{B}{A} \right\rbrace \\ 
= & \frac{2}{\pi} \arctan \frac{A}{A^2 + B^2 - B} = 1
\end{align*}
because $A^2 + B^2 - B =0$. In the general case we can verify that \eqref{distr:qwA} integrates to unity  for all $\gamma >0$ in the following manner
\begin{align*}
\int_0^1 Pr\{ \hat{R}^\gamma(t) \in dw \} = & \int_0^1 \frac{d}{dw} \int_\frac{1-w}{w}^\infty Pr\{ R^\gamma(T_t) \in dr \}\\
= & \int_0^1 \frac{1}{w^2} Pr \left\lbrace  R^\gamma(T_t) \in d\left( \frac{1-w}{w} \right) \right\rbrace \\
= & \left[ \frac{1-w}{w} =y \right] = \int_0^\infty Pr\{ R^\gamma(T_t) \in dy \} = 1 .
\end{align*}
\end{os}

\begin{os}
\normalfont
As a by-product of our calculations we show that
\begin{equation}
\int_{-\frac{1}{t}}^t \frac{(t-y)^{\gamma -1}}{(1+y^2)^\frac{\gamma +1}{2}} dy =  \left( \frac{t}{1+t^2} \right)^\frac{1-\gamma}{2} \frac{1}{2} B\left(\frac{\gamma}{2}, \frac{1}{2} \right).
\end{equation}
We start from the relationship in the proof of Theorem \ref{prTe} integrated in $(0,1)$
\begin{align*}
& \frac{2t}{\sqrt{\pi}} \frac{\Gamma\left( \frac{\gamma +1}{2} \right)}{\Gamma\left( \frac{\gamma}{2} \right)} \int_0^1 \frac{1}{w^2} \left( \frac{1-w}{w} \right)^{\gamma -1} \frac{dw}{\left[ t^2 + \left( \frac{1-w}{w} \right)^2 \right]^\frac{\gamma +1}{2}}\\
= & \frac{2 t^\frac{1-\gamma}{2}}{\sqrt{\pi}} \frac{\Gamma\left( \frac{\gamma +1}{2} \right)}{\Gamma\left( \frac{\gamma}{2} \right)} \int_0^1 (1-w)^{\gamma -1} \left[ \frac{\frac{t}{1+t^2}}{\left(w - \frac{1}{1+t^2} \right)^2 + \frac{t^2}{1+t^2} } \right]^\frac{\gamma +1}{2}.
\end{align*}
Now put $\frac{1-w}{w}=y$ in the first integral and $w-\frac{1}{1+t^2}=\frac{yt}{1+t^2}$ in the second and obtain
\begin{equation*}
t \int_0^\infty \frac{y^{\gamma -1} \, dy}{(t^2 + y^2)^\frac{\gamma +1}{2}} = \left( \frac{t}{1+t^2} \right)^\frac{\gamma +1}{2} \int_{-\frac{1}{t}}^t \frac{(t-y)^{\gamma -1}}{(1+y^2)^\frac{\gamma +1}{2}} dy.
\end{equation*}
Thus the identity 
\begin{equation*}
\int_{-\frac{1}{t}}^t \frac{(t-y)^{\gamma -1}}{(1+y^2)^\frac{\gamma +1}{2}} dy = \frac{(1+t^2)^\frac{\gamma -1}{2}}{t^\frac{\gamma -3}{2}} \int_0^\infty \frac{y^{\gamma -1}}{(t^2 + y^2)^\frac{\gamma +1}{2}}dy
\end{equation*}
emerges. Therefore,
\begin{equation}
\int_{-\frac{1}{t}}^t \frac{(t-y)^{\gamma -1}}{(1+y^2)^\frac{\gamma +1}{2}} dy =  \left( \frac{t}{1+t^2} \right)^\frac{1-\gamma}{2} \frac{\Gamma\left( \frac{\gamma}{2} \right) \Gamma\left( \frac{1}{2} \right)}{2\,\Gamma\left( \frac{\gamma +1}{2} \right)}.
\end{equation}
\end{os}

\begin{os}
\normalfont
Another result related to distribution \eqref{dist:coincides} states that
\begin{equation}
Pr\left\lbrace \frac{t^3}{t^2 + |R^\gamma(T_t)|^2} \in dr \right\rbrace /dr = \frac{1}{t} \frac{1}{B(\frac{\gamma}{2}, \frac{1}{2})} \left( \frac{r}{t} \right)^{\frac{1}{2}-1} \left( 1-\frac{r}{t} \right)^{\frac{\gamma}{2} -1}
\label{dist:ppeQ}
\end{equation}
for $0 < r < t$, $\gamma>0$. It sufficies to evaluate the following integral
\begin{align*}
Pr \left\lbrace \frac{t^3}{t^2 + |R^\gamma(T_t)|^2} \in dr \right\rbrace / dr = \frac{\Gamma\left( \frac{\gamma +1}{2}\right)}{\Gamma\left( \frac{\gamma}{2} \right)}  \frac{d}{dr} \int_{\sqrt{\frac{t^3 - rt^2}{r}}}^\infty \frac{2t}{\sqrt{y}} \frac{y^{\gamma -1}}{(y^2 + t^2)^\frac{\gamma +1}{2}} dy .
\end{align*}
For $\gamma=1$ from \eqref{dist:ppeQ} one obtains the law of sojourn time on $(0,\infty)$ of Brownian motion and the even-order pseudo processes, while for odd values of $\gamma$ the distribution of the sojourn time on the half-line for odd-order pseudo processes emerges (for $\gamma=3$ see \cite{Ors91}, $\gamma=2n+1$, $n>2$ see \cite{LCH03}).
\end{os}

\begin{os}
\normalfont
For $\gamma=n$ the process $R^n(T_t)$, $t>0$ can be represented as
\begin{equation}
R^n(T_t) = \sqrt{\sum_{j=1}^n B^2_j(T_t)}, \quad t>0
\label{RnP}
\end{equation}
where $B_j(t)$, $t>0$, $j=1,2,\ldots , n$ are independent Brownian motions and the r.v.'s $B_j(T_t)$, $t>0$ possess Cauchy distribution. Therefore \eqref{RnP} represents the Euclidean distance of an $n$-dimensional Cauchy random vector $\left(C_1(t), \ldots , C_2(t) \right)$, $t>0$.
\end{os}

\begin{os}
\normalfont
For us it is relevant to study the distribution of
\begin{equation*}
S(t)= \frac{1}{R^\gamma(T_t)}, \quad t>0.
\end{equation*}
After some calculation we find that
\begin{equation}
P\{ S(t) \in dr \} = dr \, \frac{2t}{\sqrt{\pi}} \frac{\Gamma \left( \frac{\gamma+1}{2} \right)}{\Gamma\left( \frac{\gamma}{2} \right)} \frac{1}{(1 + r^2 t^2)^\frac{\gamma +1}{2}}, \quad r,t>0.
\label{dist:S}
\end{equation}
We note that for $t=\frac{1}{\sqrt{n}}$ , $\gamma=n$ the density \eqref{dist:S} coincides with a folded \textit{t}-distribution with $n$ degrees of freedom and its density takes the form 
\begin{equation}
f(r;n)=\frac{2}{\sqrt{\pi n}} \frac{\Gamma\left( \frac{n+1}{2} \right)}{\Gamma\left( \frac{n}{2} \right)} \frac{1}{\left( 1+ \frac{r^2}{n} \right)^\frac{n+1}{2}}, \quad r>0.
\label{dist:St}
\end{equation}
For $n=1$, the density \eqref{dist:St} coincides with a folded Cauchy and coincides with \eqref{dist:coincides} for $\gamma=1$ and at time $t=1$.
\end{os}

\begin{te}
The probability law \eqref{dist:coincides} of the process $R^\gamma(T_t)$, $t>0$ is a solution to the following equation
\begin{equation}
-\frac{\partial^2 q}{\partial t^2} = \left( \frac{\partial^2}{\partial r^2} - (\gamma-1) \frac{\partial}{\partial r} \frac{1}{r} \right) q , \quad r,t>0, \; \gamma >0.
\label{eqTt}
\end{equation}
\end{te}
\begin{proof} It is easy to check that the density of the first-passage time
\begin{equation}
f(s,t)=\frac{t e^{-\frac{t^2}{2s}}}{\sqrt{2 \pi s^3}}, \quad s,t>0
\end{equation}
satisfies the equation
\begin{equation}
\frac{\partial^2 f}{\partial t^2} = 2\frac{\partial f}{\partial s}, \quad s,t>0.
\label{eqf}
\end{equation}
In view of \eqref{eqf} we have that
\begin{align*}
\frac{\partial^2 q}{\partial t^2} = & \int_0^\infty p(r,s) \frac{\partial^2 }{\partial t^2} f(s,t) ds = 2 \int_0^\infty p(r,s) \frac{\partial }{\partial s} f(s,t) ds\\
= &2p(r,s) f(s,t) \Bigg|_{s=0}^{s=\infty} - 2\int_0^\infty \frac{\partial}{\partial s} p(r,s) f(s,t) ds\\
= &- \left( \frac{\partial^2}{\partial r^2} - (\gamma-1) \frac{\partial}{\partial r} \frac{1}{r} \right) q .
\end{align*}
\end{proof}

\begin{os}
\normalfont
For $\gamma=1$ the Bessel process coincides with the reflected Brownian motion so that $R^1(T_t)$, $t>0$ is a reflected Brownian motion stopped at the random time $T_t$ and therefore becomes a folded Cauchy process. It is easy to prove that the Cauchy density 
\begin{equation}
q(r,t)=\frac{t}{\pi (t^2 + r^2)}, \quad r,t>0
\end{equation}
solves the Laplace equation and this agrees with \eqref{eqTt}.
\end{os}

By inverting the role of the Bessel process and of the first-passage time we obtain a new process somehow related to $R^\gamma(T_t)$, $t>0$ which we denote by
\begin{equation}
T_{R^\gamma(t)} = \inf \{ s:\, B(s)=R^\gamma(t) \}
\label{proc:TR}
\end{equation}
where $T_{R^\gamma(t)}$, $t>0$ is the first instant where a Brownian motion $B$ first attains the level $R^\gamma(t)$, $t>0$ (and $R^\gamma$ is independent from $B$). The probability law of \eqref{proc:TR} becomes
\begin{align}
Pr\left\lbrace T_{R^\gamma(t)} \in dx \right\rbrace /dx = & \int_0^\infty \frac{s e^{-\frac{s^2}{2x}}}{\sqrt{2 \pi x^3}} 2 \frac{s^{\gamma -1} e^{-\frac{s^2}{2t}}}{(2t)^\frac{\gamma}{2} \Gamma\left( \frac{\gamma}{2} \right)} ds \label{dist:DDD} \\
= & \frac{2}{2^\frac{\gamma+1}{2} t^\frac{\gamma}{2} \Gamma\left( \frac{\gamma}{2} \right) \sqrt{\pi} x^\frac{3}{2}} \int_0^\infty s^\gamma e^{-\frac{s^2}{2} \left(\frac{1}{x}+\frac{1}{t} \right)} ds\nonumber \\
= & \frac{\sqrt{t} x^{\frac{\gamma}{2} - 1}}{\sqrt{\pi} (x+t)^\frac{\gamma +1}{2}} \frac{\Gamma\left( \frac{\gamma+1}{2} \right)}{\Gamma\left( \frac{\gamma}{2} \right)}, \quad x,t>0,\; \gamma >0.\nonumber
\end{align}
It can be checked that the distribution \eqref{dist:DDD} integrates to unity. 

\begin{os}
\normalfont
It can be easily checked that the following relationship holds
\begin{equation}
\sqrt{T_{R^\gamma(t^2)}} \stackrel{i.d.}{=} R^\gamma(T_t), \quad t>0.
\label{wrty}
\end{equation}
From \eqref{wrty} one can also infer that
\begin{equation}
T_{R^\gamma(t)} \stackrel{i.d.}{=} \left( R^\gamma ( T_{\sqrt{t}}) \right)^2, \quad t>0.
\label{dist:qzc}
\end{equation}
In particular, for $\gamma=1$ the result \eqref{dist:qzc} says that
\begin{equation}
T_{R^1(t)} \stackrel{i.d.}{=} \left( B ( T_{\sqrt{t}}) \right)^2 \stackrel{i.d.}{=} \left( C(\sqrt{t}) \right)^2, \quad t>0.
\end{equation}
%The governing equation of the distribution \eqref{dist:DDD} reads
%\begin{equation}
%\left\lbrace 2 t\, \frac{\partial^2}{\partial t^2} + \frac{\partial}{\partial t}\right\rbrace g  = - \frac{1}{2^2}\left\lbrace x \frac{\partial^2}{\partial x^2} - (\gamma -1)\frac{\partial}{\partial x} \right\rbrace g, \quad x,t>0.
%\end{equation}
\end{os}

\begin{os}
\normalfont
We note that the probability density \eqref{dist:DDD} for $\gamma=1$, $t=1$ coincides with the ratio of two independent first-passage times through level $1$ of two independent Brownian motions. In other words we have that
\begin{equation}
Pr\left\lbrace T_{R^1(1)} \in dw \right\rbrace = Pr \left\lbrace W_{1/2} \in dw \right\rbrace = \frac{1}{\pi} \frac{w^{-\frac{1}{2}}}{w + 1} dw, \quad w>0
\end{equation}
where $W_{1/2} = T^1_{1/2} /T^2_{1/2}$ and $T^1_{1/2}$, $T^2_{1/2}$ are the first-passage times of $B^1$ and $B^2$ through level $1$ and are stable r.v.'s of order $1/2$. 
\end{os}

\begin{os}
\normalfont
The last statement is a special case of the following result. For stable positive, independent r.v.'s $T^1_\nu$, $T^2_\nu$ with Laplace transform 
\begin{equation}
E e^{-\lambda T_\nu} = e^{-\lambda^\nu}, \quad \lambda >0, 0<\nu<1
\label{Lap:transform}
\end{equation}
the ratio $T^1_\nu /T^2_\nu \stackrel{i.d.}{=} W_\nu$ where
\begin{equation}
Pr\left\lbrace W_\nu \in dw \right\rbrace /dw = \frac{\sin \pi \nu}{\pi} \frac{w^{\nu -1}}{1 + w^{2\nu} + 2w^\nu \cos \pi \nu}, \quad w>0
\label{eqBella}
\end{equation}
(see e.g. \cite{ChYor03}, \cite{Lanc}). We provide a simple and self-contained proof of this result based on Mellin transforms. Let us consider two independent, positively skewed stable r.v.'s $Y_1$, $Y_2$ of degree $\nu>0$ with Laplace transform \eqref{Lap:transform}. The density function $g$ of the ratio $Y_1 /Y_2$ reads
\begin{equation}
g(w) = \int_0^\infty x f(x)f(xw)dx, \quad w \geq 0
\label{dist:ratio}
\end{equation}
where $f$ is the density of $Y_1$ and $Y_2$. The Mellin transform of \eqref{dist:ratio} becomes 
\begin{align}
(\mathcal{M} g) (\eta) = & \int_0^\infty w^{\eta -1} g(w) dw = \int_0^\infty w^{\eta -1} \left\lbrace \int_0^\infty x f(x)f(xw)dx \right\rbrace dw\nonumber \\
= & (\mathcal{M} g) (\eta) \, (\mathcal{M} g) (2-\eta), \quad 0 < \Re \{ \eta \} < 1.
\end{align}
In order to write $f$ we resort to the Fourier transform (see e.g. \cite{Zol86}) and write
\begin{equation}
f(x) = \frac{1}{2\pi} \int_{-\infty}^{\infty} \exp\left\lbrace - i \beta x - |\beta |^\nu \cos \frac{\pi \nu}{2} \left(1 - i \textrm{sgn}(\beta) \tan \frac{\pi \nu}{2} \right)  \right\rbrace d\beta.
\end{equation}
This is because
\begin{align}
E e^{i\beta Y} = & \left[ \textrm{by } \eqref{Lap:transform} \right] = \exp\left\lbrace -(-i\beta)^\nu \right\rbrace = \exp\left\lbrace -|\beta |^\nu e^{-i \frac{\pi \nu}{2} \textrm{sgn}(\beta)} \right\rbrace \nonumber\\
= & \exp\left\lbrace -|\beta|^\nu \cos \frac{\pi \nu}{2} \left( 1- i \textrm{sgn}(\beta) \tan \frac{\pi \nu}{2} \right) \right\rbrace .
\end{align}
We now evaluate
\begin{align*}
(\mathcal{M} f)(\eta) = & \int_0^\infty x^{\eta -1} \left\lbrace \frac{1}{2\pi} \int_{-\infty}^{\infty} \exp\left\lbrace - i \beta x - |\beta |^\nu e^{-i \frac{\pi \nu}{2} \textrm{sgn}(\beta)}  \right\rbrace d\beta \right\rbrace dx\\
= & \frac{\Gamma(\eta)}{2\pi} \left\lbrace \int_{-\infty}^\infty |\beta |^{-\eta} \exp\left\lbrace -i\frac{\pi \eta}{2}\textrm{sgn}(\beta) - |\beta |^\nu e^{-i \frac{\pi \nu}{2} \textrm{sgn}(\beta)} \right\rbrace  d\beta \right\rbrace \\
= & \frac{\Gamma(\eta)}{2\pi} \left\lbrace \int_{0}^\infty \beta^{-\eta} \exp\left\lbrace -i\frac{\pi \eta}{2} - \beta^\nu e^{-i \frac{\pi \nu}{2}} \right\rbrace  d\beta \right .\\
+ & \left . \int_{0}^\infty \beta^{-\eta} \exp\left\lbrace i \frac{\pi \eta}{2} - \beta^\nu e^{i \frac{\pi \nu}{2}} \right\rbrace  d\beta \right\rbrace.
\end{align*}
Since
\begin{align*}
&  \int_{0}^\infty \beta^{-\eta} \exp\left\lbrace -i\frac{\pi \eta}{2} - \beta^\nu e^{-i \frac{\pi \nu}{2}} \right\rbrace d\beta + \int_{0}^\infty \beta^{-\eta} \exp\left\lbrace i \frac{\pi \eta}{2} - \beta^\nu e^{i \frac{\pi \nu}{2}} \right\rbrace d\beta \\
& = e^{-i\frac{\pi \eta}{2}} \int_{0}^\infty \beta^{-\eta} \exp\left\lbrace - \beta^\nu e^{-i \frac{\pi \nu}{2}} \right\rbrace d\beta + e^{i \frac{\pi \eta}{2}} \int_{0}^\infty \beta^{-\eta} \exp\left\lbrace - \beta^\nu e^{i \frac{\pi \nu}{2}} \right\rbrace d\beta\\
& = \frac{e^{-i\frac{\pi \eta}{2}}}{\nu} \int_{0}^\infty z^{\frac{1- \eta}{\nu}-1} e^{-z} e^{-i\frac{\eta \pi}{2} + i \frac{\pi}{2}}dz + \frac{e^{i\frac{\pi \eta}{2}}}{\nu} \int_{0}^\infty z^{\frac{1- \eta}{\nu}-1} e^{-z} e^{i\frac{\eta \pi}{2} - i \frac{\pi}{2}}dz \\
& = \frac{1}{\nu}e^{-i\pi \eta + i \frac{\pi}{2}} \Gamma\left( \frac{1-\eta}{\nu} \right) + \frac{1}{\nu} e^{i\pi \eta - i \frac{\pi}{2}} \Gamma\left( \frac{1-\eta}{\nu} \right)\\
& = \frac{2}{\nu} \cos\left( \pi \eta - \frac{\pi}{2} \right) \Gamma\left( \frac{1-\eta}{\nu} \right) = \frac{2}{\nu} \sin\left( \pi \eta \right) \Gamma\left( \frac{1-\eta}{\nu} \right)
\end{align*}
we have
\begin{equation}
(\mathcal{M}f)(\eta) = \frac{\Gamma(\eta)}{\pi \nu}\sin\left( \pi \eta \right) \Gamma\left( \frac{1-\eta}{\nu} \right), \quad 0 < \Re\{ \eta \} < 1.
\end{equation}
Analogously
\begin{align}
(\mathcal{M}f&)(2-\eta) = \frac{\Gamma(2-\eta)}{\pi \nu}\sin\left( 2\pi - \pi \eta \right) \Gamma\left( \frac{\eta-1}{\nu} \right) \\
= & - \frac{\Gamma(2-\eta)}{\pi \nu}\sin\left( \pi \eta \right) \Gamma\left( \frac{\eta-1}{\nu} \right) =  \frac{\Gamma(1-\eta)}{\pi}\sin\left( \pi \eta \right) \Gamma\left( \frac{\eta-1}{\nu} +1 \right) \nonumber
\end{align}
and thus the Mellin transform of $g$ becomes
\begin{align}
(\mathcal{M}g)(\eta) = \frac{\Gamma(\eta) \Gamma(1-\eta)}{\pi^2 \nu}\Gamma\left( \frac{1-\eta}{\nu} \right) \sin^2\left( \pi \eta \right) \Gamma\left( \frac{\eta-1}{\nu} +1 \right)
\end{align}
and in light of $\Gamma(x) \Gamma(1-x) = \pi / \sin \pi x$ we obtain
\begin{align}
(\mathcal{M}g)(\eta) = & \frac{1}{\pi \nu}\Gamma\left( \frac{1-\eta}{\nu} \right)  \Gamma\left(1- \frac{1- \eta}{\nu} \right) \sin\left( \pi \eta \right) = \frac{\sin\left( \pi \eta \right)}{\nu \, \sin\left( \pi \frac{1-\eta}{\nu}\right)}.
\end{align}
Let us take the density
\begin{equation}
h(x) = \frac{x \sin \pi \nu}{\pi (x^2 +1 + 2x \cos \nu \pi)}, \quad x \geq 0, \; 0 < \nu < 1
\end{equation}
and evaluate its Mellin transform
\begin{equation}
(\mathcal{M}h)(\eta) = \frac{\sin \pi \nu}{\pi}\int_0^\infty \frac{x^\eta \, dx}{x^2 +1 + 2x \cos \nu \pi}.
\end{equation}
We have that
\begin{align*}
 \int_0^\infty & x^{\eta -1} h(x) dx\\
= & \frac{\sin \pi \nu}{\pi}\int_0^\infty \frac{x^\eta}{(x+e^{-i\pi \nu}) (x+e^{i\pi \nu})}\\
= & \frac{1}{2 \pi i} \int_0^\infty x^\eta \left[ \frac{1}{(x+e^{-i\pi \nu})} - \frac{1}{ (x+e^{i\pi \nu})} \right] dx\\
= & \frac{1}{2 \pi i} \int_0^\infty x^\eta \left[ \int_0^\infty e^{-u(x+e^{-i\pi \nu})} du - \int_0^\infty e^{-u (x+e^{i\pi \nu})} du \right] dx\\
= & \frac{\Gamma(\eta +1)}{2 \pi i} \left[ \int_0^\infty \exp\left\lbrace -u e^{-i\pi \nu}\right\rbrace  \frac{du}{u^{\eta +1}} - \int_0^\infty \exp\left\lbrace -u e^{i\pi \nu}\right\rbrace  \frac{du}{u^{\eta +1}} \right] \\
= & \frac{\Gamma(\eta +1) \Gamma(-\eta)}{2 \pi i} \left[ e^{-i\pi \nu \eta} - e^{i\pi \nu \eta} \right]\\
= & - \frac{\Gamma(\eta +1) \Gamma(-\eta)}{\pi} \sin \pi \nu \eta= - \frac{\sin \pi \nu \eta}{\sin -\eta \pi} =  \frac{\sin \pi \nu \eta}{\sin \eta \pi}.
\end{align*}
Furthermore
\begin{align*}
& \int_0^\infty x^{\eta -1} \frac{x^{\nu -1} \, \sin \pi \nu}{\pi (x^{2\nu} + 1 + 2x^\nu \cos \pi \nu)} dx, \quad \Re\{ \eta \} > 1- \nu \\
= & \int_0^\infty x^{\eta -2} h(x^\nu) dx = \frac{1}{\nu} \int_0^\infty y^{\frac{\eta -1}{\nu}-1} h(y) dy = \frac{1}{\nu} \frac{\sin \pi \eta -\pi}{\sin \frac{\eta -1}{\nu} \pi} = \frac{1}{\nu} \frac{\sin \pi \eta}{\sin \frac{1 - \eta}{\nu} \pi}. 
\end{align*}
where the change of variable has been introduced in the first step. 
\end{os}

\begin{os}
\normalfont
A relationship between Mittag-Leffler functions and the distribution appearing in \eqref{eqBella} exists and reads 
\begin{equation}
\int_0^\infty  \frac{e^{-\lambda^\frac{1}{\nu} t x} \, x^{\nu -1} \, \sin \pi \nu}{\pi (x^{2\nu} + 1 + 2x^\nu \cos \pi \nu)} dx = E_{1,\nu}(-\lambda t^\nu) 
\end{equation}
with $\lambda>0$, $0<\nu < 1$.
\end{os}

We now consider the composition $Z(t) = C(S_\nu(t))$, $t>0$ where $C$ is a Cauchy process independent from the stable law $S_\nu$ with Laplace transform
\begin{equation}
E e^{-\lambda S_\nu(t)} = e^{-t \lambda^\nu}, \quad \lambda >0, \, 0<\nu <1, \; t>0.
\label{lap:tSub}
\end{equation}
We remark that for $\nu=1/2$, \eqref{lap:tSub} gives $E e^{-\lambda S_{1/2}(t)} = e^{- t \sqrt{\lambda}}$ which represents the Laplace transform of the first-passage time of Brownian motion for level $t/\sqrt{2}$, $t>0$. The probability distribution of $Z(t)$, $t>0$ reads
\begin{align}
Pr\{ Z(t) \in dz \}/dz = & \int_0^\infty \frac{t}{\pi (t^2 + s^2)} p_\nu(s, t) ds \nonumber\\
=& \frac{1}{2\pi} \int_{-\infty}^{+\infty} e^{-i \beta x - t |\beta |} \int_{0}^{\infty} p_\nu(s,t) ds \, d\beta 
\end{align}
where $p_\nu(s,t)$, $s \geq 0$, $t>0$ is the law corresponding to \eqref{lap:tSub}. Therefore, in view of \eqref{lap:tSub} we have that
\begin{equation}
Pr\{ Z(t)  \in dz \}/dz = \frac{1}{2\pi} \int_{-\infty}^{\infty} e^{-i\beta x - t |\beta |^\nu} d\beta, \quad x \in \mathbb{R},\, t>0.
\end{equation}

\section{Some generalized compositions}
We somehow generalize the previous results by considering the twice iterated Brownian first-passage time. By
\begin{equation}
I_{T}(t)=T^1_{T^2_t} = \inf \{ s_1 :\, B^1(s_1) = \inf\{s_2 : \, B^2(s_2) = t \} \}, \quad t>0
\label{proc:IT2}
\end{equation}
we mean a process which represents the first instant $T^1$ where a Brownian motion $B^1$ hits the level $T^2_t$ and $T^2_t$ represents the first instant where a Brownian motion $B^2$ (independent of $B^1$) hits level $t$. Clearly the probability density of the process \eqref{proc:IT2} reads 
\begin{equation}
Pr\{ I_{T}(t) \in dx \} / dx = \int_{0}^\infty \frac{s e^{-\frac{s^2}{2x}}}{\sqrt{2\pi x^3}} \frac{t e^{-\frac{t^2}{2s}}}{\sqrt{2\pi s^3}} ds.
\label{qawE}
\end{equation}
For the $n$-stage iterated Brownian first-passage time we have
\begin{equation}
I^n_{T}(t) = T^1_{I^{n-1}_t}, \quad t>0
\end{equation}
and the corresponding law becomes
\begin{align}
f^n(t,x) = & Pr\{ I^n_{T_t} \in dx \}/dx \label{eqfn}\\ 
= & \int_0^\infty \ldots \int_0^\infty \frac{s_1 e^{-\frac{s_1^2}{2x}}}{\sqrt{2\pi x^3}} \frac{s_2 e^{-\frac{s_2^2}{2s_1}}}{\sqrt{2\pi s_1^3}} \ldots \frac{t e^{-\frac{t^2}{2s_{n-1}}}}{\sqrt{2\pi s_{n-1}^3}} ds_1 \ldots ds_{n-1} \nonumber
\end{align}
For the $n$-times iterated Brownian first-passage time
\begin{equation}
I^n_T(t) = \inf \{ s_1:\, B^1(s_1) = \inf \{ s_2 : \, B^2(s_2) = \ldots = \inf \{s_n: \, B(s_n)=t \} \} \}
\end{equation}
we have the following theorem.
\begin{te}
The distribution \eqref{eqfn} of the process $I^n_T(t)$, $t>0$ satisfies the following p.d.e.
\begin{equation}
\frac{\partial^{2^n}}{\partial t^{2^n}} f^n(t,x) = 2^{2^n -1} \frac{\partial }{\partial x} f^n(t,x)
\label{pertaineq}
\end{equation}
and possesses Laplace transform
\begin{equation}
\int_0^\infty e^{-\lambda x} f^n(t,x) dx = \exp\left\lbrace -t \lambda^\frac{1}{2^n} \, 2^{1-\frac{1}{2^n}} \right\rbrace  .
\end{equation}
\end{te}
\begin{proof}
In view of \eqref{eqf}, by successive integrations by parts we have that
\begin{equation}
\frac{\partial^{2^n} q}{\partial t^{2^n}} f^n(t,x) =
\label{qqqqEQ}
\end{equation}
\begin{equation*}
\int_0^\infty \ldots \int_0^\infty \frac{s_1 e^{-\frac{s_1^2}{2x}}}{\sqrt{2\pi x^3}} \frac{s_2 e^{-\frac{s_2^2}{2s_1}}}{\sqrt{2\pi s_1^3}} \ldots \frac{\partial^{2^n}}{\partial t^{2^n}} \frac{t e^{-\frac{t^2}{2s_{n-1}}}}{\sqrt{2\pi s_{n-1}^3}} ds_1 \ldots ds_{n-1}=
\end{equation*}
\begin{equation*}
\int_0^\infty \ldots \int_0^\infty \frac{s_1 e^{-\frac{s_1^2}{2x}}}{\sqrt{2\pi x^3}} \frac{s_2 e^{-\frac{s_2^2}{2s_1}}}{\sqrt{2\pi s_1^3}} \ldots 2^{n-1}\frac{\partial^{2^{n-1}}}{\partial s_{n-1}^{2^{n-1}}} \frac{t e^{-\frac{t^2}{2s_{n-1}}}}{\sqrt{2\pi s_{n-1}^3}} ds_1 \ldots ds_{n-1}=
\end{equation*}
\begin{equation*}
2^{2^n -1} \frac{\partial}{\partial x} f^n(t,x)
\end{equation*}
The Laplace transform of \eqref{eqfn} has a very nice structure and reads
\begin{equation}
\int_0^\infty e^{-\lambda x} f^n(t,x)dx = 
\label{qqqL}
\end{equation}
\begin{equation*}
\int_0^\infty e^{-\lambda x} \int_0^\infty \ldots \int_0^\infty  \frac{s_1 e^{-\frac{s^2_1}{2x}}}{\sqrt{2\pi x^3}} \ldots \frac{t e^{-\frac{t^2}{2s_{n-1}}}}{\sqrt{2\pi s^3_{n-1}}}   ds_1 \ldots ds_{n-1} dx=
\end{equation*}
\begin{equation*}
\int_0^\infty \ldots \int_0^\infty e^{-s_1 \sqrt{2\lambda}} \frac{s_2 e^{-\frac{s^2_2}{2s_1}}}{\sqrt{2\pi s_1^3}} \ldots \frac{t e^{-\frac{t^2}{2s_{n-1}}}}{\sqrt{2\pi s^3_{n-1}}}   ds_1 \ldots ds_{n-1} =
\end{equation*}
\begin{equation*}
e^{-t \sqrt{2 \sqrt{\ldots \sqrt{2\lambda}}}} = \exp \left\lbrace -t \lambda^\frac{1}{2^n} 2^{1-\frac{1}{2^n}}\right\rbrace , \quad t>0\; \mu >0.
\end{equation*}
If we take the Laplace transform of equation \eqref{qqqqEQ} we get that
\begin{equation}
\frac{\partial^{2^n}}{\partial t^{2^n}} \mathcal{L}(t,\lambda) = 2^{2^n -1} \mu \mathcal{L}(t, \lambda).
\label{eqLpde}
\end{equation}
It is straightforward to realize that \eqref{qqqL} satisfies equation \eqref{eqLpde}.
\end{proof}

The $2^n$ initial conditions pertaining to \eqref{pertaineq} are given in the following form
\begin{equation}
\int_0^\infty e^{-\lambda x} \frac{d^k}{d t^k} f^n(t,x) \Big|_{t=0} dx = \left( -\lambda^\frac{1}{2^n} 2^{1-\frac{1}{2^n}} \right)^k .
\label{lapI}
\end{equation}
%By inverting \eqref{lapI} we obtain
%\begin{equation}
%\frac{\partial^k}{\partial t^k} f^n(t,x) \Big|_{t=0} = (-1)^k 2^{k - \frac{k}{2^n}} \frac{k}{2^n} \frac{x^{-\frac{k}{2^n} - 1}}{\Gamma\left( 1- \frac{k}{2^n} \right)}, \quad x>0.
%\end{equation}

\begin{os}
\normalfont
If we consider the generalization of \eqref{dist:coincides}, that is 
\begin{equation*}
R^\gamma(I^n_{T}(t)), \quad t>0
\end{equation*}
the corresponding probability law satisfies the $2^n$-th order equation
\begin{equation}
\frac{\partial^{2^n} q}{\partial t^{2^n}} = \frac{2^{2^n -1}}{2} \left( \frac{\partial^2}{\partial r^2} - (\gamma -1)\frac{\partial}{\partial r} \frac{1}{r} \right) q, \quad r>0,\, t>0.
\end{equation}
\end{os}

\begin{os}
\normalfont
The following shows that between the iterated first-passage time $I_T$ and the iterated Brownian motion there is a strict connection. Indeed the distribution \eqref{qawE} can be written as
\begin{align*}
Pr\{ I_{T}(t) \in dx \} / dx = & \int_{0}^\infty \frac{s e^{-\frac{s^2}{2x}}}{\sqrt{2\pi x^3}} \frac{t e^{-\frac{t^2}{2s}}}{\sqrt{2\pi s^3}} ds =  \frac{t}{x} \int_{0}^\infty \frac{e^{-\frac{s^2}{2x}}}{\sqrt{2\pi x}} \frac{e^{-\frac{t^2}{2s}}}{\sqrt{2\pi s}} ds\\
= & \frac{t}{x} Pr\{ B^2(|B^1(x)|) \in dt \}/dt, \quad x>0,\; t>0.
\end{align*}
\end{os}

We now consider some subordinated processes involving the first-passage time of a Brownian motion with drift $\mu$, say $T^\mu_t$, $t>0$. To make this topic as self-contained as possible we present a derivation of the distribution of $T^\mu_t$, $t>0$. 
\begin{lem}
\label{lem:MAX}
The maximal distribution of a Brownian motion with drift $\mu$ reads
\begin{align}
Pr\left\lbrace \max_{0 \leq s \leq t} B_\mu(s)  > \beta \right\rbrace = & Pr\left\lbrace T^\mu_\beta \leq t \right\rbrace, \quad \beta >0,\, t>0 \\
= & e^{2\mu \beta} \int_\frac{\beta+\mu t}{\sqrt{t}}^\infty \frac{e^{-\frac{w^2}{2}}}{\sqrt{2\pi}} dw + \int_\frac{\beta-\mu t}{\sqrt{t}}^\infty \frac{e^{-\frac{w^2}{2}}}{\sqrt{2\pi}}dw. \nonumber
\end{align}
\end{lem}
\begin{proof}
For $\beta > \alpha$ we have that
\begin{align}
Pr \left\lbrace \max_{0 \leq s \leq t} B_\mu (s) > \beta , \; B_\mu (t)  \leq \alpha \right\rbrace = & Pr\left\lbrace T^\mu_\beta \leq t, \; B_\mu(t) \leq \alpha \right\rbrace \label{dist:ins2} \\
= & E \left\lbrace  \mathbb{I}_{\left[ T^\mu_\beta \leq t \right] } \, Pr\left\lbrace B_\mu(t) \leq \alpha \Big| B_\mu (T^\mu_\beta) \right\rbrace  \right\rbrace \nonumber.
\end{align}
Furthermore, 
\begin{align}
& Pr\left\lbrace B_\mu(t) \leq \alpha \Big| \; B_\mu(T^\mu_\beta) \right\rbrace \label{dist:insert} \\
= & \int_{-\infty}^\alpha \frac{dw}{\sqrt{2\pi (t - T^\mu_\beta)}} \exp\left\lbrace - \frac{(w - \beta -\mu(t-T^\mu_\beta))^2}{2(t - T^\mu_\beta)} \right\rbrace\nonumber \\
= & \int_{-\infty}^\alpha  \frac{dw\; e^{2\mu(w - \beta)}}{\sqrt{2\pi (t - T^\mu_\beta)}} \exp\left\lbrace - \frac{(w - \beta + \mu(t-T^\mu_\beta))^2}{2(t - T^\mu_\beta)} \right\rbrace \nonumber \\
= & \left[ 2\beta -w=y \right] = \int_{2\beta -\alpha}^\infty 	\frac{dy\;e^{2\mu(\beta - y)}}{\sqrt{2\pi(t-T^\mu_\beta)}} \exp\left\lbrace -\frac{(\beta - y + \mu(t -T^\mu_\beta))^2 }{2(t - T^\mu_\beta)}\right\rbrace  \nonumber \\
= & E\left\lbrace e^{2\mu (\beta - B_\mu(t))} \mathbb{I}_{\left[ B_\mu(t) > 2\beta -\alpha \right]} \Big| B_\mu(T^\mu_\beta) \right\rbrace \nonumber.
\end{align}
Thus by inserting \eqref{dist:insert} into \eqref{dist:ins2} we get that
\begin{align*}
& Pr \left\lbrace \max_{0 \leq s \leq t} B_\mu (s) > \beta , \; B_\mu (t)  \leq \alpha \right\rbrace =\\
& E \left\lbrace \mathbb{I}_{\left[ T^\mu_\beta \leq t \right] }\, E\left\lbrace e^{2\mu (\beta - B_\mu(t))} \mathbb{I}_{\left[ B_\mu(t) > 2\beta -\alpha \right]} \Big| B_\mu(T^\mu_\beta) \right\rbrace  \right\rbrace = \\
& E \left\lbrace 2^{2\mu (\beta - B_\mu(t))} \mathbb{I}_{T^\mu_\beta \leq t} \mathbb{I}_{\left[ B_\mu(t) > 2\beta -\alpha \right]} \right\rbrace = E \left\lbrace e^{2\mu (\beta - B_\mu(t))} \mathbb{I}_{\left[ B_\mu(t) > 2\beta -\alpha \right]}\right\rbrace .
\end{align*}
In the light of the calculation above we can write that
\begin{align*}
& Pr\left\lbrace \max_{0 \leq s \leq t} B_\mu(s)  > \beta \right\rbrace = \\
& Pr\left\lbrace \max_{0 \leq s \leq t} B_\mu(s) > \beta, \; B_\mu(t) \leq \beta \right\rbrace + Pr\left\lbrace B_\mu(t) > \beta \right\rbrace=\\
& E \left\lbrace e^{2\mu (\beta - B_\mu(t))} \mathbb{I}_{\left[ B_\mu(t) > \beta \right]}\right\rbrace + Pr\left\lbrace B_\mu(t) > \beta \right\rbrace =\\
& e^{2\mu \beta} \int_\beta^\infty \frac{e^{-\frac{(w+\mu t)^2}{2t}}}{\sqrt{2 \pi t}} dw + \int_\beta^\infty \frac{e^{-\frac{(w-\mu t)^2}{2t}}}{\sqrt{2 \pi t}} =\\
& e^{2\mu \beta} \int_\frac{\beta+\mu t}{\sqrt{t}}^\infty \frac{e^{-\frac{w^2}{2}}}{\sqrt{2\pi}} dw + \int_\frac{\beta-\mu t}{\sqrt{t}}^\infty \frac{e^{-\frac{w^2}{2}}}{\sqrt{2\pi}}dw= Pr\left\lbrace T^\mu_\beta \leq t \right\rbrace .
\end{align*}
This concludes the proof.
\end{proof}
We list here some consequences of Lemma \ref{lem:MAX} :
\begin{itemize}
\item [i)] the density of $T^\mu_\beta$ reads
\begin{equation}
q(\beta, t) = Pr\left\lbrace T^\mu_\beta \in dt \right\rbrace /dt = \beta \frac{e^{-\frac{(\beta - \mu t)^2}{2t}}}{\sqrt{2 \pi t^3}}, \quad \beta >0,\; t>0 
\end{equation} 
\item [ii)] the density $q(\beta ,t)$ satisfies  equation
\begin{equation}
\frac{\partial^2 q}{\partial \beta^2} - 2\mu \frac{\partial q}{\partial \beta} = 2 \frac{\partial q}{\partial t}, \quad \beta >0,\, t>0
\label{pde:FPTmu}
\end{equation}
\item [iii)]  the following relations hold
%\begin{align}
%& \int_0^\infty t\, Pr\{ T^\mu_\beta \in dt \} = \beta \, e^{\beta \mu} \int_0^\infty e^{-\frac{\mu^2 t}{2}} \frac{e^{-\frac{\beta^2}{2t}}}{\sqrt{2\pi t}} dt =  \beta \, e^{\beta \mu}  \frac{e^{-\beta \mu}}{\mu} = \frac{\beta}{\mu} \label{mV}\\
%& \int_0^\infty t^{\eta} P\{ T^\mu_\beta \in dt \} =  \frac{\sqrt{2} \beta \, e^{\mu \beta}}{\sqrt{\pi}} \left( \frac{t}{\mu} \right)^{\eta - 1/2} K_{\eta -1/2} \left( \mu \beta \right), \\
%& \int_0^\infty e^{-\lambda t} Pr\{T^\mu_\beta \in dt \} = e^{\mu \beta} \int_0^\infty e^{-\frac{t}{2} (2\lambda + \mu^2) } \beta  \frac{e^{-\frac{\beta^2}{2t}}}{\sqrt{2 \pi t^3}} dt = e^{\beta \mu} e^{-\beta \sqrt{2\lambda + \mu^2}} \label{rTg}
%\end{align}
\begin{align}
& \int_0^\infty t\, Pr\{ T^\mu_\beta \in dt \} =  \beta \, e^{\beta \mu}  \frac{e^{-\beta |\mu |}}{|\mu |}  = \left\lbrace \begin{array}{ll} \frac{\beta}{|\mu |} e^{-2 \beta |\mu |}, & \mu <0\\ \frac{\beta}{\mu}, & \mu >0 \end{array} \right . \label{mV}\\
& \int_0^\infty t^{\eta} P\{ T^\mu_\beta \in dt \} =  \frac{\sqrt{2} \beta \, e^{\mu \beta}}{\sqrt{\pi}} \left( \frac{t}{| \mu |} \right)^{\eta - 1/2} K_{\eta -1/2} \left(| \mu | \beta \right), \\
& \int_0^\infty e^{-\lambda t} Pr\{T^\mu_\beta \in dt \} = e^{\beta \mu} e^{-\beta \sqrt{2\lambda + \mu^2}} \label{rTg}
\end{align}
for $\beta>0$, $\mu \in \mathbb{R}$, $\eta > \frac{1}{2}$, $\lambda >0$.
\end{itemize}
We now present the distribution of $R^\gamma(T_t^\mu)$, $t>0$ which generalizes the result \eqref{dist:coincides}.
\begin{te}
The distribution of $R^\gamma(T_t^\mu)$, $t>0$ has the following form
\begin{align}
q_\mu(r,t) = & Pr\left\lbrace R^\gamma(T_t^\mu) \in dr \right\rbrace /dr\nonumber \\
= & \frac{4t\, e^{t\mu} r^{\gamma-1}}{2^\frac{\gamma}{2} \Gamma\left( \frac{\gamma}{2} \right) \sqrt{2\pi}} \left(\frac{\mu^2}{r^2+t^2} \right)^\frac{\gamma +1}{4} K_\frac{\gamma + 1}{2} \left( |\mu | \sqrt{r^2+t^2}\right) \label{dist:RTmu}
\end{align}
with $r>0$, $t>0$, $\mu \in \mathbb{R}$.
\end{te}
\begin{proof}
\begin{align}
q_\mu(r, t) = & \int_{0}^\infty \frac{2 \, r^{\gamma - 1} e^{-\frac{r^2}{2s}}}{(2s)^\frac{\gamma}{2} \Gamma\left( \frac{\gamma}{2} \right)} Pr\{ T^\mu_t \in ds \} \label{eRt}\\
= & \frac{2t\, e^{t\mu} r^{\gamma-1}}{2^\frac{\gamma}{2} \Gamma\left( \frac{\gamma}{2} \right) \sqrt{2\pi}} \int_0^\infty s^{-\frac{\gamma+1}{2}-1} e^{-\frac{r^2+t^2}{2s} - \frac{s \mu^2}{2}} ds \nonumber\\
= & \frac{4t\, e^{t\mu} r^{\gamma-1}}{2^\frac{\gamma}{2} \Gamma\left( \frac{\gamma}{2} \right) \sqrt{2\pi}} \left(\frac{\mu^2}{r^2+t^2} \right)^\frac{\gamma +1}{4} K_\frac{\gamma + 1}{2} \left(| \mu |\sqrt{r^2+t^2}\right).
\end{align}
Result \eqref{dist:RTmu} emerges on applying formula \eqref{formula:K}.
\end{proof}

\begin{os}
\normalfont
By applying the asymptotic formula for the Modified Bessel function $K_\nu$
\begin{equation}
K_\nu(x) \approx \frac{2^{\nu-1} \Gamma(\nu)}{x^\nu}, \quad \textrm{for} \quad x \to 0^+
\end{equation}
(see  p 136 \cite{LE} or p. 929 \cite{GR}) we have that
\begin{align*}
q_0(r, t) = & \frac{4}{2^\frac{\gamma}{2} \Gamma\left( \frac{\gamma}{2} \right) \sqrt{2\pi}} \frac{t\, r^{\gamma-1}}{(r^2+t^2)^\frac{\gamma +1}{4}}  \frac{2^\frac{\gamma +1}{2} \Gamma\left(\frac{\gamma +1}{2} \right)}{(r^2+t^2)^\frac{\gamma +1}{4}}\\
= & \frac{2}{ \Gamma\left( \frac{\gamma}{2} \right) \sqrt{\pi}} \frac{t\, r^{\gamma-1}}{(r^2+t^2)^\frac{\gamma +1}{2}} \Gamma\left(\frac{\gamma +1}{2} \right), \quad r,t>0.
\end{align*}
\end{os}

The equation governing the distribution $q_\mu(r,t)$ is given in the next theorem.
\begin{te}
The distribution \eqref{eRt} solves the following p.d.e.
\begin{equation}
\left( 2\mu \frac{\partial}{\partial t} - \frac{\partial^2}{\partial t^2}\right) q_\mu = \left( \frac{\partial^2}{\partial r^2} - (\gamma - 1) \frac{\partial}{\partial r} \frac{1}{r} \right) q_\mu, \quad r,t>0, \quad \mu \geq 0.
\label{pdeWWE}
\end{equation}
\end{te}
\begin{proof}
We apply the derivatives $\frac{\partial^2}{\partial t^2} - 2\mu \frac{\partial}{\partial t}$ to distribution
\begin{equation}
q_\mu(r,t) = \int_0^\infty p(r,s) p_\mu(s,t) ds, \quad r\geq 0, \, t>0, \, \mu \geq 0.
\end{equation}
We readily have that (in light of \eqref{pde:FPTmu})
\begin{align*}
&\left( \frac{\partial^2}{\partial t^2} - 2\mu \frac{\partial}{\partial t} \right) q_\mu(r,t) = \int_0^\infty p(r,s) 2 \frac{\partial}{\partial s} p_\mu(s,t) ds\\
= & -2 \int_0^\infty \frac{\partial p}{\partial s}(r,s) p_\mu(s,t) ds = \left[ \eqref{pde:FPTmu} \right] = - \left( \frac{\partial^2}{\partial r^2} - (\gamma - 1)\frac{\partial}{\partial r}\frac{1}{r} \right) q_\mu (r,t).
\end{align*}
\end{proof}
For $\mu=0$ in \eqref{pdeWWE} one retrieves equation \eqref{eqTt}.

\section{Compositions of hyperbolic Brownian motions on the Poincar\'e half-space}

We consider the classical model of hyperbolic space represented by the Poincar\'e half-space $H^+_2=\{x,y:\, y>0 \}$ with the metric
\begin{equation}
ds^2=\frac{dx^2 + dy^2}{y^2}.
\end{equation}
The hyperbolic Brownian motion on $H^+_2$ is the diffusion process with generator $\mathcal{H}$ defined as
\begin{equation}
\mathcal{H}_2 = \frac{y^2}{2} \left\lbrace \frac{\partial^2}{\partial x^2} + \frac{\partial^2}{\partial y^2} \right\rbrace 
\end{equation}
and its transition function is the solution to the Cauchy problem
\begin{eqnarray}
\left\lbrace \begin{array}{l} \frac{\partial u}{\partial t} = \mathcal{H}_2 u, \quad x \in \mathbb{R}, \; y>0 \\ 
u(x,y,0) = \delta(y-1) \delta(x) . \end{array} \right . 
\label{ProbLL}
\end{eqnarray}
It is convenient to study the hyperbolic Brownian motion in terms of hyperbolic coordinates $(\eta, \alpha)$ where $\eta$ is the hyperbolic distance of $(x,y)$ from the origin $(0,1)$ of $H^+_2$. In explicit terms $(\eta, \alpha)$ and $(x,y)$ are related by
\begin{equation}
\cosh \eta = \frac{x^2+y^2+1}{2y}.
\end{equation}
Furthermore, $\alpha$ is connected with $(x,y)$  by
\begin{equation}
\tan \alpha = \frac{x^2 + y^2 -1}{2x}.
\end{equation}
We note that the formulas transforming $(x,y)$ into $(\eta, \alpha)$ are
\begin{equation}
\left\lbrace \begin{array}{ll} 
x=\frac{\sinh \eta \cos \alpha}{\cos \eta - \sinh \eta \sin \alpha}, & \eta >0\\
y=\frac{1}{\cosh \eta - \sinh \eta \sin \alpha }, & -\frac{\pi}{2} < \alpha < \frac{\pi}{2}
\end{array} \right .
\end{equation}
Some details on these formulas can be found in \cite{RW00}, \cite{VO08}. The Cauchy problem \eqref{ProbLL} can be converted into hyperbolic coordinates as follows
\begin{equation}
\frac{\partial u}{\partial t} = \frac{1}{2} \left[ \frac{1}{\sinh \eta} \frac{\partial}{\partial \eta} \left( \sinh \eta \frac{\partial}{\partial \eta} \right) u + \frac{1}{\sinh^2 \eta} \frac{\partial^2 u}{\partial \alpha^2} \right], \quad \eta >0, \, t>0
\label{ProbLLL}
\end{equation}
subject to the initial condition
\[ u(\eta, \alpha, 0) = \delta(\eta) \quad \textrm{for all } \quad \alpha \in [0,2\pi). \]
If we concentrate on the distribution of the hyperbolic distance of the Brownian motion particle from the origin we disregard the dependence in \eqref{ProbLLL} from $\alpha$ and study
\begin{equation}
\left\lbrace \begin{array}{l} \frac{\partial u}{\partial t} = \frac{1}{2} \frac{1}{\sinh \eta} \frac{\partial}{\partial \eta} \left( \sinh \eta \frac{\partial}{\partial \eta} \right) u\\u(\eta, 0)=\delta(\eta). \end{array}  \right . 
\label{ProbLLLL}
\end{equation}
It is well-known that the solution to \eqref{ProbLLLL} has the following form
\begin{equation}
q(\eta , t) = \frac{e^{-\frac{t}{8}}}{\sqrt{\pi t^3}} \int_\eta^\infty \frac{\varphi \, e^{-\frac{\varphi^2}{2t}}}{\sqrt{\cosh \varphi - \cosh \eta}} d\varphi, \quad \eta >0,\, t>0.
\label{LAA}
\end{equation}
If we pass from problem \eqref{ProbLLLL} to
\begin{equation}
\left\lbrace \begin{array}{l}
\frac{\partial u}{\partial t} = \frac{1}{\sinh \eta} \frac{\partial}{\partial \eta} \left( \sinh \eta \frac{\partial}{\partial \eta} \right) u\\
u(\eta, 0)=\delta(\eta) \end{array} \right .
\end{equation}
(by means of the change of variable $t^\prime=t/2$) we obtain a somewhat different distribution which reads
\begin{equation}
q(\eta , t^\prime) = \frac{e^{-\frac{t^\prime}{4}}}{\sqrt{\pi (2t)^3}} \int_\eta^\infty \frac{\varphi \, e^{-\frac{\varphi^2}{4t^\prime}}}{\sqrt{\cosh \varphi - \cosh \eta}} d\varphi, \quad \eta >0, \, t>0
\label{LAAA}
\end{equation}
In the first paper \cite{GV59} (and also in the subsequent literature) the factor $1/2$ does not appear and the heat kernel is \eqref{LAAA} (up to some constants). Of course the probability density $p_2(\eta, t)$ is given by
\begin{equation}
p_2(\eta, t) = \sinh \eta \, q(\eta , t), \qquad \eta >0,\; t>0.
\label{QWaaa}
\end{equation}
A detailed derivation of \eqref{LAA} and \eqref{LAAA} is given in \cite{LaoO07}.\\
We note that the probability distribution \eqref{QWaaa}
\begin{equation}
p_2(\eta , t) = \frac{\sinh \eta \, e^{-\frac{t}{8}}}{\sqrt{\pi t^3}} \int_\eta^\infty \frac{\varphi \, e^{-\frac{\varphi^2}{2t}}}{\sqrt{\cosh \varphi - \cosh \eta}} d\varphi, \quad \eta >0,\; t>0
\label{dist:BH2}
\end{equation}
solves the adjoint equation
\begin{equation}
\frac{\partial p_2}{\partial t} = \frac{1}{2} \left\lbrace \frac{\partial^2 p_2}{\partial \eta^2} - \frac{\partial}{\partial \eta} \left( \textrm{coth}\, \eta \; p_2 \right) \right\rbrace , \quad \eta >0,\, t>0
\label{distPp2}
\end{equation}
For the distribution \eqref{dist:BH2} further characterizations are possible. Indeed, we can rewrite $p_2(\eta , t)$ as follows
\begin{equation}
p_2(\eta, t) = \sqrt{2} e^{-\frac{t}{8}} \, E \left\lbrace \mathbb{I}_{\left[T_t > \eta \right] } \frac{\sinh \eta}{\sqrt{\cosh T_t - \cosh \eta}} \right\rbrace, \quad \eta >0,\, t>0
\end{equation}
where the mean is taken with respect to the distribution of $T_t$, $t>0$, which is the first-passage time of standard Brownian motion. Moreover, 
\begin{equation}
p_2(\eta , t) = \frac{e^{-\frac{t}{8}}}{\sqrt{2 \pi t^3}} E \left\lbrace  \mathbb{I}_{\left[ R^2(t) > \eta \right]} \frac{\sinh \eta}{\sqrt{\cosh R^2(t) - \cosh \eta}} \right\rbrace, \quad \eta >0,\, t>0
\end{equation}
where $R^2(t)$, $t>0$ is the $2$-dimensional Bessel process described above and the mean-value is taken with respect to the distribution of Bessel process $R^2$.\\
We give now an alternative form of $p_2(\eta ,t)$ in terms of the Euclidean distance. Indeed, the distribution of $B^{hp}(t)$, $t>0$ in $H^+_2$ can be written as
\begin{align}
p_2(\eta, t) = & -2\frac{e^{-\frac{t}{8}}}{\sqrt{\pi t}} \frac{d}{d \eta} \int_\eta^\infty \frac{\varphi \, e^{-\frac{\varphi^2}{2t}}}{t} \sqrt{\cosh \varphi - \cosh \eta}\, d\varphi \label{AAAAzx}\\
= & - 2\frac{e^{-\frac{t}{8}}}{\sqrt{\pi t}} \frac{d}{d \eta} E\left\lbrace \mathbb{I}_{\left[ R^2(t) > \eta \right]} \,  \sqrt{\cosh R^2(t) - \cosh \eta } \right\rbrace, \quad \eta >0,\, t>0 \nonumber.
\end{align}
If we take an Euclidean right triangle with one cathetus of length $\sqrt{\cosh \eta}$ and hypotenuse $\sqrt{\cosh \varphi}$, then $\sqrt{\cosh \varphi - \cosh \eta}$ represents the length of the second cathetus. Therefore the integrals above represent the length of the second cathetus weighted by means of the probability  distribution of the Bessel process in the plane. Thus, the formula \eqref{AAAAzx} hightlights the relation between the distribution of the hyperbolic distance in $H^+_2$ and the corresponding Euclidean distance in $\mathbb{R}^2$. We can recognize the additional factor of \eqref{AAAAzx} as a gamma distribution with parameters $1/2$, $1/8$.\\

We now examine the hyperbolic Brownian motion $B^{hp}(t)$, $t>0$ stopped at the first-passage time $T_t$, $t>0$ of an independent standard Brownian motion $B$ defined as $T_t = \inf \{s:\, B(s)=t \}$. In other words we study the process
\begin{equation}
J_2(t)=B_2^{hp}(T_t), \quad t>0
\label{proc:J2}
\end{equation}
with distribution
\begin{align}
p_{J_2}(\eta, t) = & \int_0^\infty p_2(\eta, s) \frac{t\, e^{-\frac{t^2}{2s}}}{\sqrt{2\pi s^3}} ds, \qquad \eta >0,\, t>0 \label{dist:J2} \\
= & t\, \sinh \eta \int_0^\infty \int_\eta^\infty \frac{\varphi \, e^{-\frac{\varphi^2+ t^2}{2s}}}{\sqrt{\cosh \varphi - \cosh \eta}} \frac{e^{-\frac{s}{8}}}{s^3 \sqrt{2} \pi} ds d\varphi \nonumber \\
= & \frac{t\, \sinh \eta}{\sqrt{2^3} \pi} \int_\eta^\infty \frac{\varphi \, d\varphi}{\sqrt{\cosh \varphi - \cosh \eta}} \frac{1}{(\varphi^2 + t^2)} K_2\left(\frac{1}{2} \sqrt{\varphi^2 + t^2} \right) \nonumber
\end{align}
where we have used formula \eqref{formula:K}. The functions $K_\nu$ are related by the following recursive formulas
\begin{equation}
K_{\nu+1}(x) = K_{\nu-1}(x)+2 \frac{\nu}{x} K_\nu(x)
\end{equation}
and thus
\begin{equation}
K_2(x) = K_0(x) + \frac{2}{x} K_1(x).
\end{equation}
The function $K_\nu$ is the so-called Modified Bessel function of the second order. In analogy with the representation \eqref{AAAAzx} we can give the following expression for the distribution of hyperbolic Brownian motion stopped at $T_t$, $t>0$
\begin{equation*}
p_{J_2}(\eta ,t) = - \frac{1}{\sqrt{2^3}} \frac{d}{d \eta}  E\left\lbrace C(t)  \sqrt{\cosh C_t - \cosh \eta}\, K_2\left( \frac{1}{2}\sqrt{\left( C(t) \right)^2 + t^2} \right)  \mathbb{I}_{[C(t) > \eta]} \right\rbrace
\end{equation*}
with $\eta >0$, $t>0$, where $C(t)$, $t>0$ is a Cauchy process. For the distribution \eqref{dist:J2} we can state the following result.
\begin{te}
The process \eqref{proc:J2} has distribution \eqref{dist:J2}, say $p_{J_2}=p_{J_2}(\eta, t)$ which solves the following Cauchy problem
\begin{equation}
-\frac{\partial^2}{\partial t^2} p_{J_2} = \left( \frac{\partial^2}{\partial \eta^2} - \frac{\partial}{\partial \eta} \frac{1}{\tanh \eta} \right) p_{J_2}, \quad q(\eta, 0) = \delta(\eta), \quad \eta, t>0.
\end{equation}
\end{te}
\begin{proof}
The distribution of \eqref{proc:J2} can be written as
\begin{equation*}
p_{J_2}(\eta, t) = \int_0^\infty p_2(\eta) g(s,t) ds
\end{equation*}
where $g(s,t)$ is the distribution of $T_t$, $t>0$. In view of \eqref{distPp2} we have therefore that
\begin{align*}
- \frac{\partial^2}{\partial t^2} p_{J_2} = & - \int_{0}^\infty p_2(\eta, s) 2 \frac{\partial}{\partial s} g(s,t) ds\\
= & 2\int_0^\infty g(s,t) \frac{\partial}{\partial s} p_2(\eta, s) ds = \frac{\partial^2}{\partial \eta^2}p_{J_2} - \frac{\partial}{\partial \eta} \left( \frac{1}{\tanh \eta} p_{J_2} \right)
\end{align*}
\end{proof}

The hyperbolic Brownian motion on $H_3^+=\{ x,y,z; z>0 \}$ is the diffusion with generator
\begin{equation}
\mathcal{H}_3 = z^2 \left\lbrace \frac{\partial^2}{\partial x^2} + \frac{\partial^2}{\partial y^2} + \frac{\partial^2}{\partial z^2}\right\rbrace - z \frac{\partial}{\partial z}.
\end{equation}
The distribution of the hyperbolic distance of Brownian motion in $H_3^+$ possesses the form 
\begin{align}
p_3(\eta, t) = & k_3(\eta, t) \, \sinh^2 \eta  =  \frac{\sinh \eta \, e^{-t}}{2 \sqrt{\pi t^3}} \eta \, e^{-\frac{\eta^2}{4t}}, \quad \eta>0, \; t>0.
\label{sIIIo}
\end{align}
where $k_3(\eta, t)$ is the heat kernel. The kernel
\begin{equation}
k_3(\eta, t) = \frac{e^{-t}}{2 \sqrt{\pi t^3}} \frac{\eta \, e^{-\frac{\eta^2}{4t}}}{\sinh \eta}, \quad \eta >0, \; t>0
\end{equation}
is the solution to
\begin{equation}
\left\lbrace \begin{array}{l}
\frac{\partial u}{\partial t} = \frac{1}{\sinh^2 \eta } \frac{\partial}{\partial \eta} \left( \sinh^2 \eta \frac{\partial}{\partial \eta}\right) u\\ u(\eta, 0)=\delta(\eta). \end{array}
\right .
\label{eqLL0}
\end{equation}
By means of the transformation $t=t^\prime /2$ equation \eqref{eqLL0} is converted into
\begin{equation}
\frac{\partial u}{\partial t} = \frac{1}{2} \frac{1}{\sinh^2 \eta } \frac{\partial}{\partial \eta} \left( \sinh^2 \eta \frac{\partial}{\partial \eta}\right) u
\end{equation}
and formula \eqref{sIIIo} takes the form 
\begin{align}
p_3(\eta ,t) = & 2\frac{\sinh \eta e^{-\frac{t}{2}}}{\sqrt{2 \pi t^3}} \eta e^{-\frac{\eta^2}{2t}} = \frac{\sinh \eta}{\eta} e^{-t/2} q_3(\eta, t), \quad \eta >0,\, t>0. \label{dgjk}
\end{align}
The law $q_3(\eta, t)$ is the distribution of the $3$-dimensional Bessel process or the Euclidean distance on $\mathbb{R}^3$. Alternative form for the distribution $p_3(\eta, t)$ involve the first-passage time of a standard Brownian motion
\begin{align}
p_3(\eta, t) = & 2 \sinh \eta e^{-\frac{t}{2}} Pr \{ T_\eta \in dt \} /dt =  2\frac{\sinh \eta}{e^\eta} Pr \{ T^1_\eta \in dt \} /dt
\end{align}
where $T_\eta=\inf \{ s: \, B(s)=\eta \}$, $T^1_\eta = \inf\{s: \, B^1(s) = \eta \}$ with $B^1$ denoting a Brownian motion with drift equal to $1$ and $B$ a Brownian motion without drift. It is well-known that the probability law $q_3(\eta, t)$ appearing in \eqref{dgjk} can be expressed in terms of distribution of one dimensional standard Brownian motion $B(t)$, $t>0$ since
\begin{equation}
q_3(\eta, t) d\eta = 2 Pr\{ \max_{s} B(s) \in d\eta \} - Pr \{ B(t) \in d\eta \}, \quad \eta >0,\, t>0. 
\end{equation}
From \eqref{dgjk} we can obtain that
\begin{equation}
E\left\lbrace \frac{\sinh R^3(t)}{R^3(t)} \right\rbrace = e^{\frac{t}{2}}, \quad t>0
\label{qazL}
\end{equation}
because on integrating w.r.t. $\eta$ both members of \eqref{dgjk} we obtain
\begin{equation}
1= \int_0^\infty p_3(\eta, t) d\eta = e^{-\frac{t}{2}} \int_0^\infty \frac{\sinh \eta}{\eta} q_3(\eta, t) d\eta
\end{equation}
where $q_3(\eta , t)$ is the distribution of $R^3(t)$, $t>0$ which is $3$-dimensional Bessel process. Moreover, distribution \eqref{sIIIo} solves the p.d.e.
\begin{equation}
\frac{\partial p_3}{\partial t} = \frac{\partial^2 p_3}{\partial \eta^2} - 2 \frac{\partial}{\partial \eta} \left( \frac{1}{\tanh \eta \eta} \, p_3 \right)
\end{equation}
which is the adjoint of the operator appearing in \eqref{eqLL0}.\\

The process $J_3(t)$, $t>0$ obtained by composing the $3$-dimensional hyperbolic Brownian motion $B^{hp}_3(t)$, $t>0$  with $T_t=\inf\{ s:\, B(s)=t \}$ with $B$ independent from $B^{hp}_3$, has a probability law equal to
\begin{align}
p_{J_3}(\eta, t) = &  \int_0^\infty p_3(\eta, s) \frac{t \, e^{-\frac{t^2}{2s}}}{\sqrt{2 \pi s^3}} ds =  \eta \, \sinh \eta \int_0^\infty e^{-t} \frac{e^{-\frac{\eta^2}{4s}}}{2 \sqrt{\pi} s^\frac{3}{2}} \frac{t \, e^{-\frac{t^2}{2s}}}{\sqrt{2 \pi s^3}} ds \nonumber \\
= & \frac{2 \sqrt{2}}{\pi} \frac{\eta \, t \, \sinh \eta}{\left( \eta^2 + 2 t^2 \right)} K_2\left( \sqrt{\eta^2 + 2 t^2} \right), \quad \eta >0, \, t>0.\label{dist:J3}
\end{align}
For the governing equation of \eqref{dist:J3} we present the following result.
\begin{te}
The distribution of $J_3(t)$, $t>0$, say $p_{J_3}=p_{J_3}(\eta, t)$,  solves
\begin{equation}
-\frac{\partial^2}{\partial t^2} p_{J_3} = \left( \frac{\partial^2}{\partial \eta^2} - 2 \frac{\partial}{\partial \eta} \frac{1}{\tanh \eta} \right) p_{J_3}, \quad \eta, t>0
\end{equation}
subject to the initial condition
\begin{equation}
p_{J_3}(\eta, 0) = \delta(\eta).
\end{equation}
\end{te}
\begin{proof}
Distribution \eqref{dist:J3} can be written as
\begin{equation*}
p_{J_3}(\eta, t) = \int_0^\infty p_3(\eta) g(s,t) ds
\end{equation*}
where $g(s,t)$ is the distribution of $T_t$, $t>0$. We have
\begin{align*}
- \frac{\partial^2}{\partial t^2} p_{J_3} = & - \int_{0}^\infty p_3(\eta, s) 2 \frac{\partial}{\partial s} g(s,t) ds\\
= & 2\int_0^\infty g(s,t) \frac{\partial}{\partial s} p_3(\eta, s) ds = \left( \frac{\partial^2}{\partial \eta^2} - 2 \frac{\partial}{\partial \eta}\frac{1}{\tanh \eta}\right) p_{J_3}
\end{align*}
\end{proof}

\end{document}